\PassOptionsToPackage{dvipsnames}{xcolor}
\documentclass{siamart1116}

\usepackage{amsmath}
\usepackage{amsfonts,amssymb}
\usepackage{graphicx}
\usepackage{bm}
\graphicspath{{./Figures/}}
\usepackage{color}
\usepackage{ulem}
\usepackage{multirow}
\usepackage{multicol}
\usepackage{algorithm}
\usepackage{algpseudocode}
\usepackage{algorithmicx}
\usepackage{pgfplots}

\usepackage[showonlyrefs]{mathtools}

\usepackage{cite} % order citations [1-3,5]

\DeclareMathAlphabet{\mathbf}{OT1}{cmr}{bx}{it}

\newcommand{\vb}{{\mathbf B}}
\newcommand{\vbv}{{\mathbf b}}
\newcommand{\vc}{{\mathbf C}}
\newcommand{\vcv}{{\mathbf c}}

\newcommand{\vf}{{\mathbf f}}

\newcommand{\vq}{{\mathbf Q}}

\newcommand{\vu}{{\mathbf U}}

\newcommand{\vw}{{\mathbf W}}
\newcommand{\vx}{{\mathbf X}}
\newcommand{\vxv}{{\mathbf x}}
\newcommand{\vy}{{\mathbf Y}}

\newcommand{\spK}{{\cal K}}

\renewcommand{\d}{\,\mathrm{d}}

\newcommand{\C}{\mathbb{C}}
\newcommand{\Cn}{\mathbb{C}^n}
\newcommand{\Cnn}{\mathbb{C}^{n \times n}}

\newcommand{\Cnl}{\mathbb{C}^{n \times \ell}}

\newcommand{\dmu}{\d\mu(t)}

\newcommand{\tildeU}{{U}_m}
\newcommand{\tildeV}{{V}_m}
\newcommand{\tildeX}{{X}_m}
\newcommand{\tildeG}{{G}_m}
\newcommand{\tildeH}{{H}_m}
\newcommand{\tildeE}{{E}_m}
\newcommand{\tildeF}{{F}_m}

\newcommand{\calU}{\mathcal{U}}
\newcommand{\calV}{\mathcal{V}}

\newcommand{\calA}{\mathcal{A}}
\newcommand{\calB}{\mathcal{B}}
\newcommand{\calN}{\mathcal{N}}
\newcommand{\calD}{\mathcal{D}}

\newcommand{\bbE}{\mathbb{E}}
\newcommand{\bbD}{\mathbb{D}}
\newcommand{\barC}{\overline{\mathbb{C}}}

\numberwithin{theorem}{section}

\DeclareMathOperator{\colspan}{colspan}
\newcommand{\spec}{\Lambda}

\DeclareMathOperator{\sign}{sign}

\newtheorem{remark}[theorem]{Remark}
\let\oldremark\remark
\renewcommand{\remark}{\oldremark\normalfont}

\newcommand{\lmin}{\lambda_{\min}}
\newcommand{\lmax}{\lambda_{\max}}
\newtheorem{example}[theorem]{Example}
\let\oldexample\example
\renewcommand{\example}{\oldexample\normalfont}

\renewcommand{\Re}{\textnormal{Re}}

\title{Low-rank updates of matrix functions II: Rational Krylov methods}
\author{Bernhard Beckermann\thanks{Laboratoire Paul Painlev\'e UMR 8524, D\'epartement de Math\'ematiques, Universit\'e de Lille, F-59655 Villeneuve d’Ascq CEDEX, France. E-mail: Bernhard.Beckermann@univ-lille1.fr. The work of Bernhard Beckermann has been supported in part by the Labex CEMPI (ANR-11-LABX-0007-01).} \and Alice Cortinovis\thanks{Institute of Mathematics, EPF Lausanne, 1015 Lausanne, Switzerland. E-mail: alice.cortinovis@epfl.ch. The work of Alice Cortinovis has been supported by the SNSF research project \emph{Fast algorithms from low-rank updates}, grant number: 200020\_178806.} \and Daniel Kressner\thanks{Institute of Mathematics, EPF Lausanne, 1015 Lausanne, Switzerland. E-mail: daniel.kressner@epfl.ch} \and Marcel Schweitzer\thanks{Mathematisch-Naturwissenschaftliche Fakult\"at, Heinrich-Heine-Universit\"at D\"usseldorf, Universit\"atsstra\ss{}e 1, 40225 D\"usseldorf, Germany. E-mail: marcel.schweitzer@hhu.de. The work of Marcel Schweitzer was partly supported by the SNSF research project \emph{Low-rank updates of matrix functions and fast eigenvalue solvers}.}}
\date{\today}
\begin{document}
\renewcommand{\thefootnote}{\fnsymbol{footnote}}
\maketitle \pagestyle{myheadings} \thispagestyle{plain}
\markboth{B. BECKERMANN, A. CORTINOVIS, D. KRESSNER AND M. SCHWEITZER}{RATIONAL KRYLOV UPDATES OF MATRIX FUNCTIONS}

\begin{abstract}
This work develops novel rational Krylov methods for updating a large-scale matrix function $f(A)$ when $A$ is subject to low-rank modifications.
It extends our previous work in this context on polynomial Krylov methods, for which we present a simplified convergence analysis. For the rational case, our convergence analysis is based on an exactness result that is connected to work by Bernstein and Van Loan on rank-one updates of rational matrix functions. We demonstrate the usefulness of the derived error bounds for guiding the choice of poles in the rational Krylov method for the exponential function and Markov functions. Low-rank updates of the matrix sign function require additional attention; we develop and analyze a combination of our methods with a squaring trick for this purpose. A curious connection between such updates and existing rational Krylov subspace methods for Sylvester matrix equations is pointed out.
\end{abstract}



%\textcolor{red}{??? DK: Alice can you please the definitions (spectral norm, Frobenius norm) at the place when they are first used? -- done ???}

\section{Introduction}\label{sec:intro}

% \begin{itemize}
% \item Alice/Marcel, can one of you please go over the entire paper and check the norms? I think there are many cases, where the Frobenius norm is measured but it is written $\|\cdot\|$ instead of $\|\cdot\|_F$.
% \end{itemize}

The need for computing matrix functions or associated quantities arises in a variety of applications, including network analysis~\cite{Benzi2020,EstradaHigham2010}, signal processing~\cite{Shuman2013}, machine learning~\cite{Stoll2019}, and differential equations~\cite{HochbruckOstermann2010}. In many of these applications, slight changes of the problem setting, such as
% \cite{ArrigoBenzi2016,ArrigoBenzi2016b}
removing a vertex in a graph or changing a parameter in a differential equation, induce a low-rank change of the matrix. In this work, we discuss new methods for updating the matrix function under such changes. Specifically, assuming that a matrix function $f(A)$ has been computed and $A$ is modified by a low-rank matrix $D$, we aim at computing the update
\begin{equation}\label{eq:low_rank_update}
f(A+D)-f(A)
\end{equation}
in a way that is cheaper than computing $f(A+D)$ from scratch. Such an update is also useful when only some quantities associated with $f(A)$, such as the trace or the diagonal entries, are of interest.

% arises in several applications. The most prominent example for such a situation is the change of communicability measures in network analysis (which are often based on the matrix exponential) when edges are added to or removed from a graph~. Other application areas include ??? add applications ??? change of boundary conditions ???
% General tool going beyond the realm of matrix functions. We give an example, updating invariant subspaces (maybe: solving matrix equations) but there are potentially many other applications ???

In~\cite{BeckermannKressnerSchweitzer2017}, we have introduced and analyzed an algorithm for efficiently approximating~\eqref{eq:low_rank_update} by projection onto polynomial Krylov subspaces.
While this algorithm often shows satisfactory convergence, especially for entire functions like the matrix exponential and matrices with a ``favorable'' spectral distribution, convergence can also be very slow in other cases. In particular, this can happen when $A$ has eigenvalues close to a singularity of $f$. A typical example is the matrix square root $A^{1/2}$ for a symmetric positive definite matrix $A$ with eigenvalues close to zero. Rational Krylov spaces can lead to much faster convergence in such situations; at least this is indicated by existing work on approximating
the action of the matrix function on a vector, $f(A)\vbv$, and solving matrix equations; see~\cite{DruskinKnizhnerman1998,Guettel2013,GuettelKnizhnerman2013,KnizhnermanSimoncini2010,Simoncini2007}.

The main goal of this paper is thus to extend the techniques of~\cite{BeckermannKressnerSchweitzer2017} to incorporate \emph{rational Krylov subspaces} and to analyze the convergence of the resulting algorithms. At the same time, we will also show that the original convergence analysis in~\cite{BeckermannKressnerSchweitzer2017} can be significantly simplified by using recent results from~\cite{CrouzeixKressner2020}.

The work of Bernstein and Van Loan in~\cite{BernsteinVanLoan2000} extends the Sherman--Morrison formula~\cite{ShermanMorrison1950} for updating matrix inverses to rational matrix functions.
In particular, Theorem 3 in~\cite{BernsteinVanLoan2000} gives an analytic expression for the update~\eqref{eq:low_rank_update} and shows that it has rank at most $m$ if $f$ is a rational function of degree $m$ and $D$ has rank one. In principle, it would be possible to exploit the good approximation properties of rational functions in the context of~\eqref{eq:low_rank_update} by first replacing $f$ with a suitable low-degree rational approximation $r$ and then using the formula from~\cite[Theorem 3]{BernsteinVanLoan2000} to approximate the update
$$f(A+D) - f(A) \approx r(A+D) - r(A).$$
We will discuss the relation of this approach to our new method in Section~\ref{sec:sm}.

% ??? START TODO LATER ???
The remainder of this paper is organized as follows. We begin by brief\/ly describing a general  subspace projection approach for the computation of the update~\eqref{eq:low_rank_update} in Section~\ref{sec:polynomial}. The particular choice of rational Krylov subspaces in this approach is then discussed in Section~\ref{sec:rational}. In addition, we show that the proposed rational Krylov method is exact when approximating updates of certain rational functions and discuss the connection of our approach to the generalized Sherman--Morrison formula for rational functions from~\cite{BernsteinVanLoan2000}. In Section~\ref{sec:convergence}, we analyze the convergence of our methods for several important matrix functions. 
%A first series of numerical experiments illustrating the performance of the method, also in comparison to the polynomial Krylov approach from~\cite{BeckermannKressnerSchweitzer2017} are reported in section~\ref{sec:experiments}. 
Afterwards, in Section~\ref{sec:sign}, we specifically focus on the matrix sign function and its peculiarities in the context of approximating low-rank updates; we conclude by showing a connection to Sylvester equations.
%We then explain how our approach can be used to update eigenvectors of a large, sparse matrix that is subject to a low-rank perturbation in section~\ref{sec:spectral_projectors}, by exploiting the intimate relation between the matrix sign function and spectral projectors and report numerical experiments to illustrate the performance of this approach.
%Concluding remarks are given in~\ref{sec:conclusions}.
% ??? END TODO LATER ???

\section{A subspace projection approach for low-rank updates}\label{sec:polynomial}

In this section, we present a general subspace projection approach for approximating the update~\eqref{eq:low_rank_update}, which includes the algorithm from~\cite{BeckermannKressnerSchweitzer2017} as well as our newly proposed algorithm.

Let $A \in \C^{n\times n}$ and $D \in \C^{n\times n}$ be such that both $f(A)$ and $f(A+D)$ are well defined. 
In the following, we describe how an approximation to   $f(A+D)-f(A)$ is extracted from two subspaces $\calU_m, \calV_m \subseteq \Cn$ of (low) dimension $m_U$ and $m_V$, respectively.
Considering orthonormal bases  $U_m, V_m$ of $\calU_m, \calV_m$, we let $G_m := U_m^\ast A U_m$ and $H_m := V_m^\ast A^\ast V_m$ denote the compressions of $A$ and $A^\ast$, respectively. We then use an approximation of the form
%\begin{equation}\label{eq:polynomial_update}
\[
f(A+D)-f(A) \approx U_mX_m(f)V_m^\ast,
\]
%\end{equation}
where $X_m(f)$ is the (1,2)-block of the (small) matrix function
\begin{equation}\label{eq:blockmatrix_krylov}
f\left(\begin{bmatrix}G_m & U_m^\ast D V_m \\ 0 & H_m^\ast + V_m^\ast D V_m\end{bmatrix}\right).
\end{equation}
In~\cite{BeckermannKressnerSchweitzer2017}, this particular choice of $X_m(f)$ was motivated by a polynomial exactness property for polynomial Krylov subspaces. We will see below, in Theorem~\ref{the:rational_exactness}, that an analogous property holds for rational Krylov subspaces. A more intuitive explanation, not tied to specific subspaces, is 
the observation~\cite[Lemma 2.2]{BeckermannKressnerSchweitzer2017} that
\begin{equation}\label{eq:fcalA}
f\left(\begin{bmatrix}A & D \\ 0 & A+D \end{bmatrix}\right) = \begin{bmatrix} f(A) & f(A+D)-f(A) \\ 0 & f(A+D)\end{bmatrix}.
\end{equation}
Note that the compression onto $\calU_m \oplus \calV_m$ of the block matrix on the left-hand side of~\eqref{eq:fcalA} corresponds to the matrix used in~\eqref{eq:blockmatrix_krylov}.

The described subspace projection approach is summarized in Algorithm~\ref{alg:polynomial_krylov}, which encompasses Algorithm 2 from~\cite{BeckermannKressnerSchweitzer2017}.

\begin{algorithm}
\caption{Subspace projection approach for approximating $f(A+D)-f(A)$}
\label{alg:polynomial_krylov}
\begin{algorithmic}[1]
%\State Choose subspaces $\calU_m,\ \calV_m \in \Cn$ with $\dim \calU_m = \dim\calV_m = m$
\State Compute orthonormal bases $U_m \in \C^{n\times m_U}$, $V_m \in \C^{n \times m_V}$ of subspaces $\calU_m, \calV_m$.
\State Compute compressions $G_m = U_m^\ast A U_m$ and $H_m =  V_m^\ast A^\ast V_m$.
\State Compute matrix function $F_m = f\left(\begin{bmatrix}G_m & U_m^\ast D V_m \\ 0 & H_m^\ast + V_m^\ast D V_m \end{bmatrix}\right).$
\State Set $X_m(f) = F_m(1\!:\!m_U,m_U+1\!:\!m_U+m_V)$.\label{line:defXm}
\State Return $U_m X_m(f) V_m^\ast$.
\end{algorithmic}
\end{algorithm}

In the Hermitian case, $A = A^\ast$ and $D = D^\ast$, it is sensible to choose $\calU_m = \calV_m$, and thus $U_m = V_m$. In turn, $G_m = H_m^\ast$ and the computation of the update simplifies. Using the relation~\eqref{eq:fcalA}, one observes that 
\begin{equation}\label{eq:Xm_hermitian_difference}
 X_m(f) = f\big(U_m^\ast (A+D) U_m\big) -  f\big( U_m^\ast A U_m \big) = f\big(G_m + U_m^\ast D U_m \big) - f(G_m).
\end{equation}

%In~\cite{BeckermannKressnerSchweitzer2017}, for the rank-one case $D = \vb\vc^\ast$, Algorithm~\ref{alg:polynomial_krylov} was proposed with Krylov subspaces $\calU_m = \spK_m(A,\vb)$ and $\calV_m = \spK_m(A^\ast,\vc)$, using the Arnoldi method (or the Lanczos method in the Hermitian case) for computing $U_m, V_m$ as well as $G_m, H_m$.

The stopping criterion proposed in~\cite{BeckermannKressnerSchweitzer2017} uses the difference of two iterates as a simple error estimator, i.e.,
\begin{equation}\label{eq:error_estimate_difference}
\|f(A+D)-U_mX_m(f)V_m^\ast\| \approx \|U_{m+d}X_{m+d}(f)V_{m+d}^\ast-U_mX_m(f)V_m^\ast\|
\end{equation}
for some small integer $d\ge 1$, where $\| \cdot \|$ denotes the spectral norm of a matrix.
When the subspaces are nested and, in turn, the orthonormal bases can be chosen to be nested (as it is, e.g., the case for Krylov subspaces and the Arnoldi method), we have
\begin{equation*}
\|U_{m+d}X_{m+d}(f)V_{m+d}^\ast-U_mX_m(f)V_m^\ast\| = \left\|X_{m+d}(f) - \left[\begin{array}{cc} X_m(f) & 0 \\ 0 & 0\end{array}\right]\right\|.
\end{equation*}
Hence, there is no need to explicitly form $U_{m+d}X_{m+d}(f)V_{m+d}^\ast$ or $U_mX_m(f)V_m^\ast$. The heuristic~\eqref{eq:error_estimate_difference} is often observed to give  fairly accurate approximations to the exact error even for small values of $d$, say $d =1$ or $d = 2$. A notable exception is when Algorithm~\ref{alg:polynomial_krylov} (almost) stagnates as $m$ increases; in this case a small value of $d$ might lead to severe underestimates; see~\cite[Section 6.2]{BeckermannKressnerSchweitzer2017} for an example.

%In the next section, we discuss the use of rational Krylov subspaces as approximation spaces in the above approach, with focus on algorithmic details. A detailed convergence analysis of the resulting methods is presented in section~\ref{sec:convergence}.

% \input{conv_polynomial.tex}

\section{Block rational Krylov subspace projection}\label{sec:rational}

In this section, we combine Algorithm~\ref{alg:polynomial_krylov} with rational Krylov subspaces. We assume that $D$ is of rank $\ell$ and can thus be written as $D = \vb\vc^\ast$ for block vectors $\vb,\vc \in \C^{n \times \ell}$ of full rank.

While a polynomial Krylov subspace with respect to $A$ and $B = [\vbv_1,\dots.\vbv_\ell]$  takes the form
\[
 \spK_m(A,\vb) = \colspan\!\big\{ \vb, A\vb,\ldots, A^{m-1}\vb \big\}
 = \spK_m(A,\vbv_1) + \dots + \spK_m(A,\vbv_\ell),
\]
the rational Krylov subspaces considered in this work take the form
\begin{align}
q_{m}(A)^{-1}\spK_m(A,\vb) = \colspan\!\big\{ q_{m}(A)^{-1} \vb, q_{m}(A)^{-1} A\vb,\ldots, q_{m}(A)^{-1} A^{m-1}\vb \big\},\label{eq:rational_ks}
%\\
%&= q_{m}(A)^{-1}\spK_m(A,\vbv_1) + \dots + q_{m}(A)^{-1}\spK_m(A,\vbv_\ell)\nonumber
\end{align}
for a polynomial $q_{m}(z) = (z-\xi_1)(z-\xi_2)\cdots(z-\xi_{m})$ of degree $m$
and fixed poles $\xi_1, \ldots, \xi_m \in \C$. Choosing some of the poles to be infinite corresponds to reducing the degree of $q_{m}$.
% I checked works by Goeckler/Grimm, Simoncini; they all start with B.
\begin{remark} \em 
When choosing one pole to be infinite, our definition~\eqref{eq:rational_ks} coincides with the subspace $q_{m-1}(A)^{-1}\spK_m(A,\vb)$ that is more commonly found in the literature; see, e.g.,~\cite{Guettel2010}. Note that $\vb\in q_{m-1}(A)^{-1}\spK_m(A,\vb)$ while this property fails to hold in general for $q_{m}(A)^{-1}\spK_m(A,\vb)$.
%, unless one of the poles is infinite. 
%In turn, Algorithm~\ref{alg:rational_arnoldi} also slightly differs from the standard rational Arnoldi method, which uses $\vb/\|\vb\|$ as the first basis vector.
One of the motivations for our choice~\eqref{eq:rational_ks} is that it nicely connects to the (generalized) Sherman--Morrison formula; see Section~\ref{sec:sm} below.
\end{remark}

%\comment{I replaced ``span'' by ``colspan'' here, because technically, span would mean elements of the form $q_m(A)^{-1}\sum_k \alpha_k A^k \vb$, which is not what we want. As ``colspan'' might not be completely clear to everyone, I added a second representation.}
Adapting the usual \emph{rational Arnoldi method}~\cite{ElsworthGuettel2020} to~\eqref{eq:rational_ks}, Algorithm~\ref{alg:rational_arnoldi} is used to compute an orthonormal basis $U_m = \big[\vu_1,\ldots,\vu_m\big]$ of $q_{m}(A)^{-1}\spK_m(A,\vb)$. 
In the case of an infinite pole $\xi_j = \infty$, line~\ref{line:rat2} of Algorithm~\ref{alg:rational_arnoldi} is replaced by 
$\vw_j \leftarrow A \vu_{j-1}$ for $j > 1$ and line~\ref{line:rat1} is replaced by $\tilde \vb \leftarrow A^{-1} \vb$ for $j = 1$.

% \textcolor{red}{??? DK: I don't think that Algorithm~\ref{alg:rational_arnoldi} is correct, e.g., $\vb/\|\tilde \vb\|$. Alice, please check -- AC: now it should be OK, but it's a ``special case'' of Algorithm 2.1 in~\cite{ElsworthGuettel2020} ???}
\begin{algorithm}
 \caption{Block Rational Arnoldi method}
\label{alg:rational_arnoldi} 
\begin{algorithmic}[1]
\State $\tilde \vb \leftarrow (A-\xi_1 I)^{-1} \vb$ \label{line:rat1}
\State $\vu_1 \leftarrow$ orthonormal basis of $\tilde \vb$.
\For{$j=2,3,\dots,m$}
	\State $\vw_j \leftarrow (A-\xi_j I)^{-1}A\vu_{j-1}$. \label{line:rat2}
	\For{$k = 1,\dots,j-1$}
		\State ${\bm{\alpha}}_{k,j-1} \leftarrow  \vu_k^* \vw_j$.
		\State $\vw_j \leftarrow \vw_j - \vu_k \bm{\alpha}_{k,j-1}$
	\EndFor
	\State $\vu_{j} \leftarrow $ orthonormal basis of $ \vw_j$.
% 	\State $\alpha_{j+1,j} \leftarrow \|\vw_j\|_2$.
% 	\State $\vu_{j+1} \leftarrow 1/\alpha_{j+1,j} \vw_j$.
\EndFor
\end{algorithmic}
\end{algorithm}

The description of Algorithm~\ref{alg:rational_arnoldi} assumes $\dim \left(q_{m}(A)^{-1}\spK_m(A,\vb)\right) = m\ell$, that is,  all block vectors $\vw_j$ have full rank. We will make this assumption from here on when discussing algorithms. Deflation techniques for removing linearly dependent columns are discussed, e.g., in~\cite[Section 6]{ElsworthGuettel2020}. 

We conclude our discussion of rational Krylov subspaces with a variation of an existing exactness result for rational matrix functions~\cite[Lemma 4.6]{Guettel2010}.
\begin{lemma} \label{lemma:exactrational}
Let $\Pi_{m-1}/q_{m}$ denote the space of all rational functions with numerator degree at most $m-1$ and denominator $q_{m}(z)$.
 Let $U_m$ be an orthonormal basis of $q_{m}(A)^{-1}\spK_m(A,\vb)$. Then
 \[
  r(A)\vb = U_m r(U_m^* A U_m) U_m^* \vb,
 \]
 provided that $r(A)$ and $r(U_m^*  A U_m)$ are well-defined.
\end{lemma}
\begin{proof}
Consider $r = p / q_m$ for arbitrary $p \in \Pi_{m-1}$.
 We start by noting that $q_{m}(A)^{-1}\spK_m(A,\vb) = \spK_m(A,\vq)$ with $\vq = q_{m}(A)^{-1}\vb$. By existing results for (polynomial) Krylov subspaces, see~\cite[Lemma 3.1]{Saad1992} and~\cite[Lemma 3.9]{Guettel2010}, which can be applied completely analogously in the block Krylov setting, we obtain
 \begin{equation} \label{eq:polyp}
  p(A)\vq = U_m p(U_m^* A U_m) U_m^* \vq,
 \end{equation}
 as well as
 \[
  U_m^* \vb = U_m^* q_{m}(A) \vq = q_m(U_m^* A U_m) U_m^* \vq.
 \]
 The latter relation is equivalent to $U_m^* \vq = q_m(U_m^* A U_m)^{-1}  U_m^* \vb$ and gives, when inserted into~\eqref{eq:polyp}, the desired relation:
 \[
  r(A)\vb = p(A)\vq = U_m p(U_m^* A U_m) q_m(U_m^* A U_m)^{-1}  U_m^* \vb = U_m r(U_m^* A U_m) U_m^* \vb.
 \]
% \textcolor{red}{Note by DK: There is a subtle problem with zero-pole cancellation. One could run into a situation where $r(U_m^*  A U_m)$ is well-defined but $q_m(U_m^* A U_m)$ is \emph{not} invertible and hence the argument does not apply (but the result should still be true). I chose to ignore this subtlety; it is also ignored by Guettel and possibly other authors.}
\end{proof}
%\comment{Should we mention that in the block case, actually more general results are possible, involving rational functions where the coefficients of the polynomial part are $\ell \times \ell$ matrices? Our result corresponds to just choosing these coefficients as multiplies of $I_\ell$, and we do not really need anything else. It might still be worth mentioning.}
%\textcolor{red}{??? DK: If you have a useful application in mind, please go ahead and insert a remark. If not, let us not put it; it adds too much clutter. ???}

\subsection{Algorithm}

For computing an approximation of $f(A+\vb\vc^\ast) - f(A)$, we utilize Algorithm~\ref{alg:rational_arnoldi} to compute orthonormal bases $U_m$, $V_m$ of rational Krylov subspaces \[\calU_m = q_{m}(A)^{-1}\spK_m(A,\vb),\qquad \calV_m = \bar q_{m}(A^\ast)^{-1}\spK_m(A^\ast,\vc),\] where
$q_{m}(z) = (z-\xi_1)\cdots(z-\xi_{m})$ and $\bar q_{m}(z) = (z-\overline{\xi}_1)\cdots(z-\overline{\xi}_{m})$ are both determined by the same set of poles $\xi_1,\ldots,\xi_{m}$. Although it is in principle possible to choose a different set of poles for $\calV_m$, we are not aware of advantages of such a choice.
Once $U_m,V_m$ have been computed,
we apply the general subspace projection approach, Algorithm~\ref{alg:polynomial_krylov}, with these bases. For ease of reference, Algorithm~\ref{alg:rationalkrylovupdate} summarizes the resulting procedure.

\begin{algorithm}
\caption{Rational Krylov subspace approximation of $f(A+\vb\vc^\ast)-f(A)$}
\label{alg:rationalkrylovupdate}
\begin{algorithmic}[1]
\State Perform $m$ steps of Algorithm~\ref{alg:rational_arnoldi} to compute an orthonormal basis $\tildeU$ of $q_{m}(A)^{-1}\spK_m(A,\vb)$ and set $\tildeG = \tildeU^\ast A \tildeU$.
\State Perform $m$ steps of Algorithm~\ref{alg:rational_arnoldi} to compute an orthonormal basis $\tildeV$ of $\bar q_{m}(A^\ast)^{-1}\spK_m(A^\ast,\vc)$ and set $\tildeH = \tildeV^\ast A^\ast \tildeV$. \label{rationalstep2}
\State Compute matrix function $\tildeF = f\left( \begin{bmatrix}
\tildeG & (U_m^\ast\vb)(V_m^\ast\vc)^\ast \\ 0 & H_m^\ast + (V_m^\ast\vb)(V_m^\ast\vc)^\ast
\end{bmatrix}\right)$.
\State Set $\tildeX(f) = \tildeF(1:m,m+1:2m)$.
\State Return $\tildeU \tildeX(f) \tildeV^\ast$.
\end{algorithmic}
\end{algorithm}

%??? Which of these points are already in the toolbox. Highlight better that this is not our contribution. ???

Several remarks concerning the implementation of Algorithm~\ref{alg:rationalkrylovupdate} are in order:
\begin{enumerate}
 \item The efficient and stable implementation of rational Arnoldi methods requires some care, including the need for reorthogonalization; it is therefore advisable to build on available toolboxes, like, e.g., the \texttt{RKToolbox} by Berljafa, Elsworth and G\"uttel~\cite{BerljafaElsworthGuettel2014}. 
 \item In contrast to the (standard) Arnoldi method, the compressed matrices $\tildeG$ and $\tildeH$ do \emph{not} contain the orthogonalization coefficients from Algorithm~\ref{alg:rational_arnoldi} explicitly. There are procedures which, possibly under additional conditions on the poles, circumvent
	 the additional computation of the products $\tildeU^\ast A \tildeU$ and $\tildeV^\ast A^\ast \tildeV$ and compute $\tildeG, \tildeH$ from $m\ell \times m\ell$ matrices containing the orthogonalization coefficients and the poles; see, e.g.,~\cite{ElsworthGuettel2020,Guettel2010,Guettel2013} for details.
% if the last poles are chosen as $\xi_m = \infty$ and $\tilde \xi_m = \infty$, the explicit computation of the products $\tildeU^\ast A \tildeU$ and $\tildeV^\ast A^\ast \tildeV$ can still be circumvented and $\tildeG, \tildeH$ can instead be computed from $m \times m$ matrices containing the orthogonalization coefficients and the poles; see, e.g.,~\cite{Guettel2010,Guettel2013} for details.
 \item %Suppose that $A$ is Hermitian, $\vc = \pm \vb$,
%  \textcolor{red}{??? In the rank-one case, these are indeed the only two possibilities but for higher ranks this is not true. Shall we be a bit more general? In principle, one can split up the rhs but this may not be as stable as treating the full Hermitian case. There is possibly a similar discussion in the Kressner/Massei/Robol paper. ???}
%  and the poles are closed under complex conjugation, that is, $\xi$ is a pole if and only if $\bar \xi$ is a pole, and both poles $\xi,\bar\xi$ have the same multiplicity. In particular, this holds when all poles are real. Then one can choose $U_m = V_m$ and the corresponding remarks for Algorithm~\ref{alg:polynomial_krylov} apply. Specifically, this allows us to make use of the simplified expression
% %\begin{equation}\label{eq:Xm_hermitian_difference_rational}
% $\tildeX(f) = f\big(\tildeG \pm U_m^\ast \vb \vb^\ast U_m \big) - f( \tildeG )$.
% %\end{equation} 

Assume that $A$ is Hermitian and the rank-$\ell$ update can be written in the form $D = \vb J \vb^*$ for some $\vb \in \C^{n \times \ell}$ and $J\in\C^{\ell \times \ell}$, that is, the columns of $D$ and $D^*$ span the same subspace of $\C^n$. In particular, this is the case when $D$ is also Hermitian. Further, let us suppose that the poles are closed under complex conjugation, that is, $\xi$ is a pole if and only if $\bar \xi$ is a pole, and both poles $\xi,\bar\xi$ have the same multiplicity. In particular, this holds when all poles are real. Then $q_m(A)^{-1}\spK_m(A, \vb) = \bar{q}_m(A^*)^{-1} \spK_m(A^*, \vb J^*)$. In turn, one can choose $V_m = U_m $ and Step~\ref{rationalstep2} in Algorithm~\ref{alg:rationalkrylovupdate} can be skipped and the corresponding remarks for Algorithm~\ref{alg:polynomial_krylov} apply. Specifically, we have the simplified expression $\tildeX(f) = f\big(\tildeG + U_m^\ast \vb J \vb^* U_m \big) - f( \tildeG )$.

\item When $A$ is Hermitian, the (standard) block Arnoldi method reduces to the  block Lanczos method~\cite{GolubUnderwood1977}. %\textcolor{red}{??? We have to be a bit more careful here in the block case. Block Lanczos? Do we have a nice recursion? ???} 
%Such short recurrences for the basis vectors of a rational Krylov subspace are only known in the presence of very few different poles.
Similarly, there exist short-term recurrences for extended Krylov subspaces, which only use the poles $0$ and $\infty$ repeatedly, see, e.g.,~\cite{DruskinKnizhnerman1998,Simoncini2007,JagelsReichel2011}. 

\item Each iteration of Algorithm~\ref{alg:rational_arnoldi} with a finite pole requires the solution of a shifted block linear system.
The efficiency of Algorithm~\ref{alg:rationalkrylovupdate} largely depends on how efficiently these linear systems can be solved. When using a direct sparse factorization such as the sparse LU factorization, it is advantageous to use only a few different poles, allowing for the frequent reuse of factorizations when poles repeat.
% , and how often the poles vary (when using a direct solver, one Cholesky factorization needs to be computed per pole). In cases where $A$ is banded with rather small bandwidth, e.g., tri- or pentadiagonal, and only a few different poles are needed for obtaining a good rational approximation of $f$, rational Krylov methods are thus particularly attractive.
In the non-Hermitian case, \emph{two} shifted linear systems---one with $A$ and one with $A^\ast$---have to be solved at each iteration of the method. It is worth pointing out that it suffices to compute only \emph{one} factorization
$A-\xi_j I = LU,$
because this immediately gives the other factorization 
$A^\ast -\overline{\xi}_j I = U^\ast L^\ast.$
%so that only the forward-backward substitution needs to be performed twice. As oftentimes the computation of $LU$ factorizations constitutes the by far largest part of the computation time in rational Krylov methods, the cost of the Hermitian and non-Hermitian method can be expected to be not as different as in the polynomial case (where the cost of the non-Hermitian algorithm is roughly twice the cost of the Hermitian one).
\end{enumerate}

\subsection{Exactness properties}
In~\cite{BeckermannKressnerSchweitzer2017}, it was shown that the polynomial Krylov subspace approximation for the update
$f(A+\vb\vc^\ast)-f(A)$
is exact when $f$ is a polynomial of a certain degree.
The following theorem extends this result to rational Krylov subspaces.
%In this section, we prove an analogue of this result for the rational case, where it turns out that exactness holds for rational functions with denominator $q_{m-1}$ and numerator degree at most $m-1$. This result will be the basis of the convergence analysis to follow.

\begin{theorem}\label{the:rational_exactness}
Given $A\in \Cnn$, $\vb,\vc \in \Cnl$ and $q_m(z) = (z-\xi_1)\cdots(z-\xi_m)$, with $\xi_1,\ldots,\xi_m\in \C$, the approximation returned by Algorithm~\ref{alg:rationalkrylovupdate} is exact for every $r \in \Pi_{m}/q_{m}$, that is,
$$r(A+\vb\vc^\ast)-r(A) = \tildeU \tildeX(r) \tildeV^\ast,$$
provided that $r(A)$, $r(A+\vb\vc^\ast)$ as well as $r(G_m)$, $r(H_m^\ast + (V_m^\ast\vb)(V_m^\ast\vc)^\ast)$ are well defined.
\end{theorem}

\begin{proof}
By the partial fraction expansion, a rational function $r\in \Pi_{m}/q_{m}$ can be decomposed as the sum of a constant and scalar multiples of terms of the form $(z-\xi_s)^{-j}$, $j \leq m_s$, where $m_s$ denotes the multiplicity of $\xi_s$. By linearity, it suffices to show exactness for each of the terms individually.  Exactness trivially holds for a constant function.

It remains to show exactness for $r_{\xi_s,j}(z) = (z-\xi_s)^{-j}$ for $j = 1,\dots,m_s$. The matrix $\tildeX(r_{\xi_s,j})$ entering the rational Krylov approximation 
%\begin{equation}\label{eq:rational_approximation_proof}
$\tildeU \tildeX(r_{\xi_s,j})\tildeV^\ast$
%\end{equation}
is given by the $(1,2)$ block of the matrix
\begin{equation}\label{eq:block_matrix_rational_proof}
\begin{bmatrix}
\tildeG-\xi_s I_m & U_m^* \vb \vc^* V_m \\
0 & \tildeH^\ast -\xi_s I_m + \tildeV^\ast\vb \vc^* V_m
\end{bmatrix}^{-j}.
\end{equation}
For $j = 1$, we directly obtain
\begin{equation} \label{eq:tildeE1}
 \tildeX(r_{\xi_s,1})= -(\tildeG-\xi_s I_m)^{-1}(U_m^* \vb \vc^* V_m)(\tildeH^\ast -\xi_s I_m + \tildeV^\ast\vb \vc^* V_m)^{-1}.
\end{equation}
For $j>1$,~\eqref{eq:block_matrix_rational_proof} yields the recursive relation
\[
\tildeX(r_{\xi_s,j}) = (\tildeG-\xi_s I_m)^{-(j-1)}\tildeX(r_{\xi_s,1}) + \tildeX(r_{\xi_s,j-1}) (\tildeH^\ast -\xi_s I_m + \tildeV^\ast\vb \vc^* V_m)^{-1}.
\]
Resolving this recursion and inserting~\eqref{eq:tildeE1} gives
\begin{eqnarray}
\tildeX(r_{\xi_s,j}) &=&\sum\limits_{k=0}^{j-1} (\tildeG-\xi_s I_m)^{-(j-1-k)}\tildeX(r_{\xi_s,1})(\tildeH^\ast -\xi_s I_m + \tildeV^\ast\vb \vc^* V_m)^{-k} \nonumber \\
&=&- \sum\limits_{k=0}^{j-1}  (\tildeG-\xi_s I_m)^{-(j-k)} (U_m^* \vb \vc^* V_m) (\tildeH^\ast -\xi_s I_m + \tildeV^\ast\vb \vc^* V_m)^{-(k+1)}. \label{eq:recursionresolved}
\end{eqnarray}
	Since $r_{\xi_s,s} \in \Pi_{m-1}/q_m$, we know from Lemma~\ref{lemma:exactrational}, that
\begin{equation*}\label{eq:rational_exactness_vector1}
\tildeU (\tildeG-\xi_s I_m)^{-d}\tildeU^* \vb = (A-\xi_s I)^{-d}\vb  \text{ for all } d = 1,\dots,m_s
\end{equation*}
and
\begin{equation*}\label{eq:rational_exactness_vector2}
\vc^\ast V_m (\tildeH^\ast-\xi_s I_m+\tildeV^\ast\vb \vc^* V_m)^{-d} \tildeV^\ast = \vc^\ast (A-\xi_s I+\vb\vc^\ast)^{-d} \text{ for all } d = 1,\dots,m_s.
\end{equation*}
Combined with~\eqref{eq:recursionresolved}, these relations yield
\begin{equation}
 \tildeU \tildeX(r_{\xi_s,j})\tildeV^\ast = - \sum\limits_{k=0}^{j-1} (A-\xi_s I)^{-(j-k)}\vb  \vc^\ast (A-\xi_s I+\vb\vc^\ast)^{-(k+1)}.\label{eq:resolvent_sum1}
\end{equation}

We now use the matrix identity
\[M^j - N^j = \sum_{k = 0}^{j-1} N^{j-1-k}(M-N)M^k,
\]
see~\cite[Proposition 3.1]{BeckermannKressnerSchweitzer2017}, with $M = (A-\xi_s I+\vb\vc^\ast)^{-1}$ and $N = (A-\xi_s I)^{-1}$. This yields
\begin{eqnarray}
& & (A-\xi_s I+\vb\vc^\ast)^{-j} - (A-\xi_s I)^{-j} \nonumber\\
&=& \sum_{k=0}^{j-1}(A-\xi_s I)^{-(j-1-k)}((A-\xi_s I+\vb\vc^\ast)^{-1}-(A-\xi_s I)^{-1})(A-\xi_s I + \vb\vc^\ast)^{-k} \nonumber  \\
&=& -\sum_{k=0}^{j-1}(A-\xi_s I)^{-(j-k)}\vb\vc^\ast(A-\xi_s I + \vb\vc^\ast)^{-(k+1)}, \label{eq:resolvent_sum2}
\end{eqnarray}
where the latter equality utilizes the second resolvent identity. Comparing~\eqref{eq:resolvent_sum2} with~\eqref{eq:resolvent_sum1}
%yields
%$$ \tildeU \tildeX(r_{\xi_i,j})\tildeV^\ast = (A-\xi_i I+\vb\vc^\ast)^{-j} - (A-\xi_i I)^{-j}$$
establishes the desired exactness property for $r_{\xi_s,j}$ for $j \leq m_s$. 
\end{proof}

\begin{remark} \label{remark:infinite} \em 
Although the statement of Theorem~\ref{the:rational_exactness} assumes the poles to be finite, the result also holds in the presence of infinite poles. To see this, let $\widetilde{m} \leq m$ be the multiplicity of $\infty$ as a pole of the rational Krylov subspace, that is, $\deg q_{m} = m-\widetilde{m}$.
We can then decompose a rational function $r\in \Pi_{m}/q_{m}$ as $r = p + \tilde r$ with $p \in \Pi_{\widetilde{m}}$ and $\tilde r \in \Pi_{m-\widetilde{m}-1}/q_{m}$.
By linearity, it suffices to show exactness for $p$ and $\tilde r$ individually. Because of $\spK_{\widetilde m}(A,\vb) \subset q_{m}(A)^{-1}\spK_m(A,\vb)$, exactness for $p$ can be shown along the lines of the proof of Theorem~3.2 in~\cite{BeckermannKressnerSchweitzer2017}. Exactness for $\tilde r$ follows directly from the proof of Theorem~\ref{the:rational_exactness}.
\end{remark}

\subsection{Connection to the Sherman--Morrison formula and its generalization to rational functions} \label{sec:sm}

It is instructive to rederive the Sherman--Morrison formula for rank-one updates from Algorithm~\ref{alg:rationalkrylovupdate}. Let $A$, $\vbv \not=0$, $\vcv\not=0$ be such that $A$ and $A+\vbv\vcv^*$ are invertible. By Theorem~\ref{the:rational_exactness}, one step of Algorithm~\ref{alg:rationalkrylovupdate} with pole $0$ should produce the exact update $(A+\vbv\vcv^*)^{-1}-A^{-1}$. In this situation, $U_1 = A^{-1} \vbv / \beta$, $V_1 = A^{-\ast} \vcv / \gamma$ with $\beta = \|A^{-1} \vbv\|$ and $\gamma = \|A^{-*}\vcv\|$. Therefore,
\[
\begin{bmatrix}
G_1 & (U_1^\ast\vbv)(V_1^\ast\vcv)^\ast \\ 0 & H_1^\ast + (V_1^\ast\vbv)(V_1^\ast\vcv)^\ast
\end{bmatrix} = \begin{bmatrix}
\frac{1}{\beta^2} \vbv^\ast A^{-\ast} \vbv &  \frac{1}{\beta\gamma} (\vbv^\ast A^{-\ast} \vbv) (\vcv^\ast A^{-\ast} \vcv) \\
0 & \frac{1}{\gamma^2} (\vcv^\ast A^{-\ast} \vcv) (1 + \vcv^* A^{-1} \vbv).
\end{bmatrix}
\]
Provided that $\vbv^\ast A^{-\ast} \vbv \not=0$, $\vcv^\ast A^{-\ast} \vcv\not=0$, and $1 + \vcv^* A^{-1} \vbv \not=0$, this matrix is invertible and the $(1,2)$ entry of its inverse is given by
$
 -\beta\gamma / (1 + \vcv^* A^{-1} \vbv).
$
Hence, Algorithm~\ref{alg:rationalkrylovupdate} returns the exact update
\[
 -\frac{\beta\gamma}{1 + \vcv^* A^{-1}  \vbv} U_1V_1^* = -\frac{A^{-1} \vbv \vcv^* A^{-1}}{1+\vcv^\ast A^{-1}\vbv}.
\]
Two observations can be made. On the one hand, the Sherman--Morrison formula is nicely reproduced by Algorithm~\ref{alg:rationalkrylovupdate}. On the other hand, two assumptions ($\vbv^\ast A^{-\ast} \vbv \not=0$, $\vcv^\ast A^{-\ast} \vcv\not=0$) need to be made that are not necessary, neither for the existence of  $(A+\vbv\vcv^*)^{-1}-A^{-1}$ nor for the validity of the Sherman--Morrison formula. Note that the violation of the conditions, $(\vbv^\ast A^{-\ast} \vbv)(\vcv^\ast A^{-\ast} \vcv)=0$, implies that the numerical range $W(A) := \{ \vxv^\ast\! A \vxv : \|\vxv\| = 1\}$ of $A$ contains $0$, a singularity of the matrix function. In general, it is not advisable to use Algorithm~\ref{alg:rationalkrylovupdate} in such situations and we will discuss in Section~\ref{sec:sign}, for a different scenario, how this can sometimes be circumvented.

In~\cite{BernsteinVanLoan2000}, Bernstein and Van Loan provide a generalization of the Sherman--Morrison formula for rational functions. The following theorem recalls their main result.

\begin{theorem}[Theorem 3 in~\cite{BernsteinVanLoan2000}] \label{thm:bernstein}
Let $r(z) = p(z)/q(z)$ with polynomials $p(z) = \sum_{i=0}^{m_p}\alpha_i z^i$ and $q(z) = \sum_{i=0}^{m_q}\beta_{i} z^i$ and set $m = \max\{m_p,m_q\}$. Let $H(\alpha)$ be the $m \times m$ Hankel matrix containing the coefficients $\alpha_i$, i.e.,
$$H(\alpha) = \begin{bmatrix}
\alpha_1 & \alpha_2 & \cdots & \alpha_{m_p} & 0& \cdots & 0 \\[-0.1cm]
\alpha_2 & & \reflectbox{$\ddots$} & \reflectbox{$\ddots$} & &   \reflectbox{$\ddots$}\\[-0.1cm]
\vdots & \reflectbox{$\ddots$} & \reflectbox{$\ddots$} & &  \reflectbox{$\ddots$}\\[-0.1cm]
\alpha_{m_p} & \reflectbox{$\ddots$} & &  \reflectbox{$\ddots$} \\[-0.1cm]
0& &  \reflectbox{$\ddots$} \\[-0.1cm]
\vdots  &  \reflectbox{$\ddots$} \\[-0.1cm]
0&&&&&& \text{\huge $0$} \\
\end{bmatrix} \in \C^{m \times m}$$
and define $H(\beta)\in \C^{m \times m}$ analogously. Suppose that $A \in \C^{n\times n}$, $\vbv \in \C^n$, $\vcv \in \C^n$ are such that 
$r(A)$ and $r(A+\vbv\vcv^\ast)$ are well defined. Set 
\begin{eqnarray*}
 K_m &=& [\vbv,A\vbv,\ldots,A^{m-1}\vbv],\\
 L_m &=& [\vcv, (A^\ast+\vcv\vbv^\ast)\vcv,\ldots,(A^\ast+\vcv\vbv^\ast)^{m-1}\vcv], \\
 Y_\alpha &=& L_mH(\alpha)^\ast,\ Y_\beta = L_mH(\beta)^\ast. 
\end{eqnarray*}
Then
\begin{equation}\label{eq:update_formula_vanloan}
r(A+\vbv\vcv^\ast)-r(A) = XY^\ast,
\end{equation}
where the $n\times m$ matrices $X,Y$ are defined by
$X = q(A)^{-1}K_m$ and $Y^\ast = Y_\alpha^\ast - M^{-1}Y_\beta^\ast(r(A)+XY_\alpha^\ast)$  with $M = I + Y_\beta^\ast X$.
\end{theorem}

Note that it is also stated in~\cite{BernsteinVanLoan2000} that the result of Theorem~\ref{thm:bernstein} can be extended to general rank-$\ell$ updates, but the technical details are omitted.

Consider a rational function $r$ of the form stated in Theorem~\ref{thm:bernstein} with $\beta_{m_q}\not=0$. Then Theorem~\ref{the:rational_exactness} and Remark~\ref{remark:infinite} state that Algorithm~\ref{alg:rationalkrylovupdate} is exact when choosing $m_q$ poles equal to the zeros of $q$ and, additionally, $\max\{m_p - m_q,0\}$ infinite poles. In turn, the low-rank updates produced by Algorithm~\ref{alg:rationalkrylovupdate} and Theorem~\ref{thm:bernstein} have the same rank and yield \emph{mathematically} the same result, up to normalization of the low-rank factors. Also, the cost of an algorithm based on Theorem~\ref{thm:bernstein} is comparable to the cost of Algorithm~\ref{alg:rationalkrylovupdate}. However, there are a number of important differences between these two approaches:

% Assuming that $r(A)$ is already known, the main computational work necessary for evaluating the update formula~\eqref{eq:update_formula_vanloan} for $m \ll n$ consists of $m-1$ matrix vector products with $A$ and $A^\ast+\vc\vb^\ast)$ each for forming the matrices $K_m$ and $L_m$, and $m_q$ linear system solves for computing $X = q(A)^{-1}K_m$. Thus, the computational cost of evaluating~\eqref{eq:update_formula_vanloan} is about the same as that of $m$ steps of a rational Krylov subspace method with $m_q$ finite poles. Thus, as already brief\/ly outlined in section~\ref{sec:intro}, one possible approach for approximating $f(A+\vb\vc^\ast)-f(A)$ would be to \emph{first} replace $f$ by an accurate enough rational approximation $r$ and then use~\eqref{eq:update_formula_vanloan} as an approximation for the desired update matrix.
% 
% Indeed, constructing a rational approximation $f\approx r \in \Pi_m/q_m$ for some $q_m \in \Pi_m$ and then using Algorithm~\ref{alg:rationalkrylovupdate} (with the roots of $q_m$ as poles) for approximating $r(A+\vb\vc^\ast$ would---in exact arithmetic---yield the same result as~\eqref{eq:update_formula_vanloan} due to the exactness property from Theorem~\ref{the:rational_exactness} with a comparable, but not lower, computational cost.
% 
% Thus, at first sight, it might seem as if our approach was no improvement over the known results from~\cite{BernsteinVanLoan2000}. This is, however, not true, as there are also several disadvantages to taking the route outlined above.
\begin{itemize}
\item Obviously, Algorithm~\ref{alg:rationalkrylovupdate} is more general as it applies to general functions while Theorem~\ref{thm:bernstein} is restricted to rational functions. As discussed in the introduction, Theorem~\ref{thm:bernstein} could still be used to address a general function $f$ by  constructing \emph{a priori} a rational approximation $r \approx f$. While Algorithm~\ref{alg:rationalkrylovupdate} also requires to choose the poles a priori, the numerator polynomial $p_m$ is determined automatically by the method. In turn, significantly less knowledge about the spectra of $A$ and $A+\vbv\vcv^*$ is needed in order to obtain effective approximations. Another advantage of Algorithm~\ref{alg:rationalkrylovupdate} is that it easily combines with existing adaptive pole selection strategies for rational Krylov methods~\cite{GuettelKnizhnerman2013}.

\item In contrast to Algorithm~\ref{alg:rationalkrylovupdate}, Theorem~\ref{thm:bernstein} makes explicit use of the \emph{non-orthogonal Krylov bases} $K_m,L_m$. These bases are prone to ill-conditioning as $m$ increases, see~\cite{Beckermann2000} for the case of a Hermitian matrix $A$ leading to numerical instabilities. Thus, when the degree of the rational function/approximation is rather high, we expect Algorithm~\ref{alg:rationalkrylovupdate} to be more accurate in the presence of round-off error.

\item Reiterating what we already observed for the classic Sherman--Morrison formula, Algorithm~\ref{alg:rationalkrylovupdate} requires two additional conditions not needed in Theorem~\ref{thm:bernstein}: $f(G_m)$, $f(H_m^\ast + (V_m^\ast\vbv)(V_m^\ast\vcv)^\ast)$ need to be well defined. Note, however, that these conditions are always met when the numerical ranges of $A$ and $A+\vbv\vcv^*$ do not contain a singularity of $f$.
%\item In the approach outlined above, the rational approximant $m$ needs to be fixed beforehand. In some situations it might not be clear a priori which degree $m$ is necessary to find a satisfactory approximation. In such a case, it is advantageous to use a rational Krylov method which increases the degree of the underlying rational approximant from one iteration to the next until the desired accuracy is reached.
%\item Related to the above issue, it might also sometimes be valuable to be able to use adaptive pole selection strategies, like, e.g., those of~\cite{GuettelKnizhnerman2013} in rational Krylov methods. This is also not possible when fixing $r_m$ a priori.
% I removed this because one could
\end{itemize}
 
Thus, we conclude that although our approach is related to the work in~\cite{BernsteinVanLoan2000}, it differs significantly in key aspects and it seems to be the preferred approach in many situations of practical interest.

%In the next section, we analyze the convergence of the presented rational Krylov methods. 
%One matrix function which we are particularly interested in from an application point of view is the matrix sign function. We therefore give a few details concerning this function at the typical behavior of (rational) Krylov methods for its approximation in the next section.

\section{Convergence analysis} \label{sec:convergence}

This section is concerned with the convergence analysis and its purpose is two-fold. We first show how the polynomial case can be treated in an elegant and, compared to our previous work~\cite{BeckermannKressnerSchweitzer2017}, much simpler fashion by using a result from~\cite{CrouzeixKressner2020}. Unfortunately, it is not clear how this technique extends to the rational case, which will therefore be treated separately in the second part.

In the following, we let
\begin{equation}\label{eq:error}
\tildeE(f) := f(A+\vb\vc^\ast) - f(A) - \tildeU\tildeX(f)\tildeV^\ast
\end{equation}
denote the error of the approximation returned by Algorithm~\ref{alg:polynomial_krylov}.

\subsection{Simpler convergence analysis for polynomial Krylov subspaces}\label{sec:conv_poly}

In this section, we consider the case in which $\calU_m$ and $\calV_m$ are block polynomial Krylov subspaces and we obtain a convergence result for Algorithm~\ref{alg:polynomial_krylov} based on polynomial approximation of the derivative of $f$; see Remark~\ref{rmk:compare} below for a comparison with the convergence analysis in~\cite{BeckermannKressnerSchweitzer2017}. 

%  \textcolor{red}{$\| A \|$ the spectral norm, and $\| A \|_F$ the Frobenius norm.} 
 The following lemma is key to our analysis; its proof uses a recent bound on the Fr\'echet derivative from~\cite{CrouzeixKressner2020}.  We recall that $W(A)$ denotes the numerical range of $A$.

\begin{lemma}\label{lemma:12block}
 Let $\mathcal{B} = \begin{bmatrix} B_{11} & B_{12} \\ 0 & B_{22} \end{bmatrix}$, let $\bbE$ be a compact convex set containing $W(B_{11})$ and $W(B_{22})$, and let $f$ be analytic in $\bbE$. Then 
 \begin{equation*}
  \| [f(\calB)]_{1,2} \|_F \le (1 + \sqrt{2})^2  \|f'\|_{\bbE} \|B_{12}\|_F,
 \end{equation*}
 where $[f(\calB)]_{1,2}$ denotes the $(1,2)$ block of $f(\calB)$ and $\| \cdot \|_{\bbE}$ denotes the supremum norm on $\bbE$.
\end{lemma}

\begin{proof}
 For $n \times n$ matrices $A$ and $B$, let $L_f(A, B)$ denote the Fr\'echet derivative of $f$ at $A$ applied to the matrix $B$ and let $L_f(A, \cdot)$ denote the corresponding linear operator represented as an $n^2 \times n^2$ matrix. By~\cite[Theorem 4.12]{Higham2008}, 
 \begin{equation*}
  f(\calB) = f(\calD) + L_f(\calD, \calN), \text{ where } \calD := \begin{bmatrix} B_{11} & 0 \\ 0 & B_{22} \end{bmatrix} \text{ and } \calN := \begin{bmatrix} 0 & B_{12} \\ 0 & 0 \end{bmatrix}.
 \end{equation*}
Because $f(\calD)$ is block diagonal, we have that 
\begin{equation*}
 \| [ f(\calB) ]_{1,2} \|_F = \| L_f(\calD, \calN) \|_F \le \| L_f(\calD, \cdot) \| \cdot \|B_{12}\|_F.
\end{equation*}
Corollary 5.1 in~\cite{CrouzeixKressner2020} states that $\| L_f(\calD, \cdot) \| \le (1 + \sqrt{2})^2 \|f'\|_{W(\calD)}$, which concludes the proof because $W(\calD)$, as the convex hull of $W(B_{11})$ and $W(B_{22})$, is contained in $\bbE$.
\end{proof}

Lemma~\ref{lemma:12block} applied to the matrix $\begin{bmatrix} A & D \\ 0 & A+D \end{bmatrix}$ from~\eqref{eq:fcalA} gives the following result, which might be of independent interest.

\begin{corollary}\label{cor:difference}
 Let $A, D \in \Cnn$, let $\bbE$ be a compact convex set containing the union of $W(A)$ and $W(A + D)$, and let $f$ be analytic in $\bbE$. Then
 \begin{equation}\label{eq:difference}
  \| f(A+D) - f(A) \|_F \le (1 + \sqrt{2})^2 \| f' \|_{\bbE} \| D \|_F.
 \end{equation}
\end{corollary}

% \begin{proof}
%  By~\eqref{eq:fcalA}, $f(A+D) - f(A) = [f(\calA)]_{1,2}$ for $\calA = \begin{bmatrix} A & D \\ 0 & A+D \end{bmatrix}$. The conclusion therefore follows directly from Lemma~\ref{lemma:12block}.
% \end{proof}

When $A$ and $D$ are Hermitian, it is well known that the inequality~\eqref{eq:difference} holds without the constant $(1 + \sqrt{2})^2$; see, e.g.,~\cite[Proposition 3.1.5]{Skripka2019}. For general diagonalizable matrices $A$ and $A+D$, Corollary 2.4 in \cite{Gil2010} states that 
\begin{equation*}
\|f(A+D) - f(A)\|_F \le \kappa_A\kappa_{A+D}  \max |f'| \cdot \|D \|_F,
\end{equation*}
where $\kappa_A$, $\kappa_{A+D}$ are the condition numbers of the eigenvector matrices of $A$ and $A+D$, respectively. The maximum of $|f^\prime|$ is taken over the convex hull of the spectra of $A+D$ and $A$. Corollary~\ref{cor:difference} instead holds for any matrix and does not feature the potentially large constant $\kappa_A\kappa_{A+D}$, at the cost of bounding $f'$ on a larger domain $\bbE$. 

We are now prepared to state a convergence result for Algorithm~\ref{alg:polynomial_krylov} when using block polynomial Krylov subspaces.

\begin{theorem}\label{thm:poly}
 Let $A\in \Cnn$ and let $f$ be analytic in a compact convex set $\bbE$ containing $W(A)$ and $W(A + \vb \vc^*)$. Let $U_m$, $V_m$ be orthonormal bases of $\calU_m = \spK_m(A, \vb)$,  $\calV_m = \spK_m(A^*, \vc)$. Then the error of Algorithm~\ref{alg:polynomial_krylov} 
 satisfies
 \begin{equation*}
  \| E_m(f) \|_F \le 2(1 + \sqrt{2})^2 \|\vb \vc^*\|_F \inf_{p \in \Pi_{m-1}} \| f'- p\|_{\bbE} .
 \end{equation*}
\end{theorem}

\begin{proof} 
The first part of the proof is the same as in Theorem 4.2 in~\cite{BeckermannKressnerSchweitzer2017}: The exactness property~\cite[Theorem 3.2]{BeckermannKressnerSchweitzer2017} -- which also holds in the block case -- implies that for all $q \in \Pi_m$ we have $E_m(f) = E_m(f-q)$, therefore
\begin{align}
    \| E_m(f) \|_F & =  \| (f-q)(A + \vb\vc^*) - (f-q)(A) - U_m X_m(f-q) V_m^* \|_F \nonumber \\
    & \le \| (f-q)(A + \vb\vc^*) - (f-q)(A) \|_F + \| U_m X_m(f-q) V_m^* \|_F \nonumber \\
    & \le \| (f-q)(A + \vb\vc^*) - (f-q)(A) \|_F + \| X_m(f-q) \|_F. \label{eq:triang}
\end{align}
Moreover, by definition (line~\ref{line:defXm} in Algorithm~\ref{alg:polynomial_krylov}), we have $X_m(f-p) = [f(\tilde{\calA})]_{1,2}$, where $\tilde{\calA} := \begin{bmatrix} U_m^* A U_m & U_m^* \vb \vc^* V_m \\ 0 & V_m^*(A + \vb \vc^*) V_m \end{bmatrix}$.
We can now use Corollary~\ref{cor:difference} to get 
\begin{equation}\label{eq:deriv1}
\|(f-q)(A + \vb \vc^*) - (f-q)(A) \|_F \le (1 + \sqrt{2})^2 \|(f-q)'\|_{\bbE} \|\vb\vc^*\|_F.
\end{equation}
 and Lemma~\ref{lemma:12block} to get
\begin{equation*}
 \|X_m(f-q)\|_F \le (1 + \sqrt{2})^2 \|(f-q)'\|_{\bbE} \| U_m^* \vb \vc^* V_m \|_F \le (1 + \sqrt{2})^2 \|(f-q)'\|_{\bbE} \| \vb \vc^*\|_F,
\end{equation*}
because of the inclusions $W(U_m^* AU_m) \subseteq W(A)$ and $W(V_m^* (A + \vb \vc^*) V_m) \subseteq W(A + \vb \vc^*)$. 
Combining these with~\eqref{eq:triang} gives the result of the theorem, because $q' \in \Pi_{m-1}$ can be chosen arbitrarily.
\end{proof}

\begin{remark}\label{rmk:compare} \em
Let us compare the result of Theorem~\ref{thm:poly} with Theorem 4.2 in~\cite{BeckermannKressnerSchweitzer2017}, which establishes the upper bound 
$2(1 + \sqrt{2}) \inf_{p \in \Pi_{m}} \| f - p\|_{\tilde \bbE}$ for the error in the non-Hermitian case.
While this bound features a somewhat smaller constant and the approximation of $f$ instead of $f^\prime$, it comes with the major disadantvage that $\tilde \bbE$ needs to contain the numerical range of $\calA = \begin{bmatrix} A & \vb \vc^* \\ 0 & A + \vb \vc^* \end{bmatrix}$, which can be critically larger than the convex hull of $W(A)$ and $W(A + \vb \vc^*)$. Indeed, there are situations~\cite[Figure 6.2]{BeckermannKressnerSchweitzer2017}  in which  $W(\calA)$ contains a singularity of $f$ (and hence the bound becomes void) but the assumptions of Theorem~\ref{thm:poly} are still satisfied.
In order to deal with these situations, specialized techniques had to be developed to address the issue (see~\cite[Section 5]{BeckermannKressnerSchweitzer2017}), which can now be bypassed by Theorem~\ref{thm:poly}. %\textcolor{red}{??? Marcel, can you please check whether it is possible to rederive Theorem 5.7 from our old paper (up to some constant) using Theorem 4.3? The major obstable is to find a polynomial approximation result for derivatives of Markov function. In the worst case, we could ``cheat'' here a little and prove such a result by ourselves mimicking the proof of Theorem 5.7 in the scalar case and just report the result here.   ???}
\end{remark}

\subsection{Convergence analysis for rational Krylov subspaces}
%\textcolor{red}{??? Replace $\gamma$ by $\omega$ to be consistent with polynomial paper? ???}

In this section, we analyze the convergence of the proposed rational Krylov subspace method for updating matrix functions, both in the Hermitian and non-Hermitian case for certain classes of functions.
%From now on, $E_m(f)$ defined in~\eqref{eq:error} denotes the error obtained by Algorithm~\ref{alg:rationalkrylovupdate}.
% \textcolor{red}{??? It would be nicer to move the definition of the error and the formula $\tildeE(f) = f(A+\vb\vb^\ast) - f(A) - \tildeU(
%     f(\tildeG+U_m^\ast\vb\vb^\ast U_m) - f(\tildeG)
%     )\tildeU^\ast $ to the beginning of the whole section, because it seems to be needed for both, the polynomial and the rational case, and we are a bit repetitive. ???}

\subsubsection{Convergence analysis in the Hermitian case} \label{sec:convergencehermitian}

We first discuss the Hermitian case, that is, $A = A^\ast$ and $D = D^\ast$. 
%and $U_m = V_m$. \textcolor{red}{??? DK: Previously, there was the factorization of $D$ here. I am not sure we need it
%(that is, the poles are closed under complex conjugation). 
%This implies $H_m^\ast=\tildeG$ and, by~\eqref{eq:Xm_hermitian_difference}, 
%we may write the error~\eqref{eq:error} as
%\begin{equation*}
%    \tildeE(f) = f(A+D) - f(A) - \tildeU(
%    f(\tildeG+U_m^\ast D U_m) - f(\tildeG)
%    )\tildeU^\ast .
%\end{equation*}
%\textcolor{red}{??? DK: Do we actually make use of this equation. We now mention this simplifcation at three different places.. ???}
 The following theorem links this error to a rational approximation problem. We omit its proof because it follows from Theorem~\ref{the:rational_exactness} in a manner entirely analogous to the proof of Theorem 4.1 in~\cite{BeckermannKressnerSchweitzer2017}.

\begin{theorem}\label{the:convergence_rational}
Let $A$ and $D = \vb J \vb^*$  be Hermitian, let the set of poles be closed under complex conjugation, and let $U_m = V_m$ be an orthonormal basis of $q_{m}(A)^{-1}\spK_m(A,\vb)$. Furthermore, let $f$ be analytic in a domain $\bbE$ containing the union of $W(A)$ and $W(A+D)$. Then the error~\eqref{eq:error} returned by Algorithm~\ref{alg:rationalkrylovupdate} satisfies
\begin{equation}\label{eq:convergence_rational_hermitian}
\|\tildeE(f)\| \leq 4 \min_{r \in \Pi_{m}/q_{m}} \|f - r\|_\bbE,
\end{equation}
where $\|\cdot\|_\bbE$ denotes the supremum norm on $\bbE$.
\end{theorem}
% \begin{proof}
% By Theorem~\ref{the:rational_exactness}, we have
% $$\tildeE(f) = \tildeE(f) - r(A+\vb\vb^\ast) + r(A) + \tildeU\tildeX(f)\tildeU^\ast = \tildeE(f-r)$$
% for any $r \in \Pi_m/q_{m-1}$. Therefore, using Crouzeix's theorem, we find
% \begin{eqnarray}
% \|\tildeE(f)\| &=& \|\tildeE(f-r)\| \nonumber\\
%                &=& \|(f-r)(A+\vb\vb^\ast) - (f-r)(A) - \tildeU\tildeX(f)\tildeU^\ast\|\nonumber\\
% 			   &\leq& \|(f-r)(A+\vb\vb^\ast\| - \|(f-r)(A)\| - \|\tildeU\tildeX(f)\tildeU\|\nonumber\\
% 			   &\leq& 2 \|f-r\|_\bbE + \|\tildeU\tildeX(f-r)\tildeU\|.\label{eq:convergence_proof_hermitian1}
% \end{eqnarray}
% Using~\eqref{eq:rational_krylov_update_hermitian} and~\eqref{eq:Xm_hermitian_difference_rational}, we have
% $$\tildeU\tildeX(f-r)\tildeU^\ast = \tildeU\big((f-r)(\tildeG+\|\vb\|^2\ve_1\ve_1^\ast)-(f-r)(\tildeG)\big)\tildeU^\ast.$$
% Thus, using the facts that $\tildeU$ has orthonormal columns, that $W(\tildeG) \subseteq W(A)$ and that $W(\tildeG+\|\vb\|^2\ve_1\ve_1^\ast) \subseteq W(A+\vb\vb^\ast)$, we can again apply Crouzeix's theorem to bound the second term in~\eqref{eq:convergence_proof_hermitian1} by $2 \|f-r\|_\bbE$. The assertion of the theorem then follows immediately.
% \end{proof}

Theorem~\ref{the:convergence_rational} allows us to derive convergence bounds for Algorithm~\ref{alg:rationalkrylovupdate} by considering rational uniform approximation problems on intervals $\bbE$ containing $ [\widetilde \lambda_{\min}, \widetilde \lambda_{\max}]$, where
\[
\widetilde \lambda_{\min} := \min \{\lambda_{\min}(A), \lambda_{\min}(A+D)  \},\quad \widetilde \lambda_{\max} := \max \{\lambda_{\max}(A), \lambda_{\max}(A+D)  \}.\] This problem has been addressed numerous times in the literature, e.g., in the context of analyzing rational Krylov subspace methods for approximating $f(A)b$; see, e.g.,~\cite{BeckermannReichel2009,Guettel2010,Guettel2013} and the references therein. In the following, we give several examples for the bounds obtained this way.

\paragraph{The exponential function}

Under the assumptions of Theorem~\ref{the:convergence_rational}, consider the exponential function $f(z)=\exp(x)$. %\textcolor{red}{??? The following sentence and the footnote need to be carefully adjusted once the Hermitian case is adjusted ???}
We will suppose in the following that the spectra of $A$ and $A+D$ (and the corresponding poles) have already been shifted\footnote{Such a shift would lead to an additional factor $\exp(\widetilde \lambda_{\max})$ in~\eqref{eq:convergence_rational_hermitian}.} such that $A$ and $A+D$ are negative semi-definite and thus one can choose $\bbE=(-\infty,0]$ in Theorem~\ref{the:convergence_rational}.

%\begin{example}\label{ex:exp1}
   From the seminal work of Gonchar and Rakhmanov~\cite{Gonchar1987} and its improvements established by Aptekarev~\cite{Aptekarev2002}
   it is known that for every integer $m$ there exists an optimal denominator $q_m\in \Pi_m$ such that
   $$
        \min_{r \in \Pi_{m}/q_{m}} \|\exp - r\|_{(-\infty,0]} \leq C \, \kappa^{-m}, \quad \kappa\approx 9.28903... .
   $$
   for some constant $C$ independent of $m$. The numerical values of the optimal poles (that is, the roots of $q_m$) are known.
   %and can be found, for certain values of $m$,  in the RK Toolbox \textcolor{red}{??? Is this 100\% true? Does it have exactly these pools by Gonchar et al.? [AC] I can't find the poles for the exponential in the toolbox, but only the ones for sqrt / invsqrt / sign  ???}
%\end{example}

We now consider the case of a single, repeated pole, which bears the advantage that only one sparse factorization needs to be computed when using a direct solver in Algorithm~\ref{alg:rationalkrylovupdate}.
%\begin{example}\label{ex:exp2}
   Andersson~\cite{Andersson1981} showed that, for $q_m(z)=(z-m/\sqrt{2})^m$,
$$
     \limsup_{m\to \infty} \Bigl( \min_{r \in \Pi_{m}/q_{m}} \| \exp - r \|_{L^\infty((-\infty,0])} \Bigr)^{1/m} = \frac{1}{1+\sqrt{2}}.
$$
%\end{example}
This agrees with observations from~\cite{MN04,vdEH06} that a well-chosen single pole $\xi$ repeated $m$ times already yields good convergence.

Strategies for choosing poles (adaptively) for finite intervals are surveyed in~\cite[Sec. 4.2]{Guettel2013}.

\paragraph{Markov functions}\label{sec_markov}
Under the assumptions of Theorem~\ref{the:convergence_rational}, let us now consider a \emph{Markov function}
%$f$ and a Hermitian positive definite matrix $A$. Markov functions can be written as
\begin{equation}\label{eq:markov}
f(x) = \int_\alpha^\beta \frac{\d\mu(z)}{x-z},
\end{equation}
where $\mu$ is a positive measure with support in the interval $[\alpha,\beta]$ with $-\infty \leq \alpha < \beta < \infty$. Important examples of Markov functions are inverse fractional powers
\begin{equation}\label{eq:inverse_fractional_powers}
    f(z)=z^{-\gamma} = \frac{\sin(\gamma\pi)}{\pi} \int_{-\infty}^0 \frac{(-x)^{-\gamma}\d x}{z-x}
\end{equation}
for $\gamma \in (0,1)$, or
\begin{equation}\label{eq:logarithm_rational}
    f(z)=\frac{1}{z}\log(1+z) = \int_{-\infty}^{-1} \frac{(-1/x)\d x}{z-x}.
\end{equation}
For more details on Markov functions and further examples we refer the reader to~\cite{BergForst1975,Henrici1977}. A detailed discussion of rational approximation of Markov functions can be found in~\cite[Section 6]{BeckermannReichel2009}. From~\cite[Theorem 6.1(b)]{BeckermannReichel2009} we quote the following estimate.

\begin{theorem}\label{the:estimate_rational_Markov}
   Let $\bbE$ be a compact convex set, symmetric with respect to the real axis, and let $f$ be a Markov function \eqref{eq:markov} such that
   \begin{equation} \label{eq_assumtion_markov}
           \beta < \omega:=\min \bbE \cap \mathbb R.
   \end{equation}
   Let $\psi$ denote the conformal map from $\barC \setminus \bbD$ onto $\barC \setminus \bbE$ normalized such that $\psi(\infty) = \infty, \psi^\prime(\infty) > 0$, where $\barC = \C \cup \{\infty\}$ denotes the extended complex plane and $\bbD$ denotes the closed unit disk, and let $\phi$ denote its inverse map from $\barC \setminus \bbE$ onto $\barC \setminus \bbD$.
   Then
   $$
        \min_{r \in \Pi_{m}/q_{m}} \| f - r \|_{L^\infty(\bbE)}
        \leq \frac{2 \| f \|_{L^\infty(\mathbb E)}}{|\phi(\beta)|} \cdot \eta_m, \quad \eta_m := \max_{x \in [\phi(\alpha),\phi(\beta)]} \frac{1}{|B_m(x)|},
   $$
   with the Blaschke product
   $$
     B_m(x) := \prod_{j=1}^m \frac{1-x\overline{\phi(\xi_j)}}{x-\phi(\xi_j)}.
   $$
\end{theorem}

For estimating $\|\tildeE(f)\|$ for Markov functions $f$, we may therefore combine Theorem~\ref{the:convergence_rational} with Theorem~\ref{the:estimate_rational_Markov} for $\bbE=[\widetilde\lambda_{\min},\widetilde\lambda_{\max}]$, as long as $\beta < \omega=\widetilde\lambda_{\min}$. In this case, explicit formulas for the conformal maps $\phi,\psi$ are available.  Noting that only the convergence factor $\eta_m$ depends on the poles $\xi_1,...,\xi_m$, it remains to derive upper bounds on $\eta_m$ for particular choices of poles.

According to~\cite[Corollary 6.4]{BeckermannReichel2009}, we may minimize $\eta_m$ among all single, repeated poles $\xi=\xi_1=...=\xi_m$ by setting
$$\sigma = \frac{\phi(\beta)-\phi(\alpha)}{\phi(\beta)\phi(\alpha)-1},\quad y_{\text{opt}} = -\frac{1}{\sigma} - \sqrt{\frac{1}{\sigma^2}-1}, \quad w = \frac{1+\phi(\alpha)y_{\text{opt}}}{\phi(\alpha)+y_{\text{opt}}},$$
resulting in the optimal pole $\xi = \psi(w)$ and $\eta_m=|y_{\text{opt}}|^{-m}$.

%\begin{example}
%We begin by considering rational Krylov subspaces that only use a single, repeated pole. They have the advantage that only one sparse factorization needs to be computed when using a direct solver in Algorithm~\ref{alg:rationalkrylovupdate}.

%Using~\cite[Corollary 6.4]{BeckermannReichel2009}, we can derive an explicit formula for the optimal repeated pole and the resulting convergence rate: The set $\bbE = [\lmin(A),\lmax(A+\vb\vb^\ast]$, where $\lmin(\cdot),\lmax(\cdot)$ denote the smallest and largest eigenvalues, fulfills the requirements of Theorem~\ref{the:convergence_rational}. Let $\psi$ denote the conformal map from $\barC \setminus \bbD$ onto $\barC \setminus \bbE$ normalized such that $\psi(\infty) = \infty, \psi^\prime(\infty) > 0$, where $\barC = \C \cup \{\infty\}$ denotes the extended complex plane and $\bbD$ denotes the closed unit disk, and let $\psi$ denote its inverse map from $\barC \setminus \bbE$ onto $\barC \setminus \bbD$. By setting $$\sigma = \frac{\phi(\beta)-\phi(\alpha)}{\phi(\beta)\phi(\alpha)-1},\quad y_{\text{opt}} = -\frac{1}{\sigma} - \sqrt{\frac{1}{\sigma^2}-1}, \quad w = \frac{1+\phi(\alpha)y_{\text{opt}}}{\phi(\alpha)+y_{\text{opt}}},$$ the optimal pole is given by $\xi = \psi(w)$. The resulting error $\|f-r_{m}\|_\bbE$ for $r_{m} \in \Pi_{m}/(z-\xi)^{m}$ decreases geometrically with rate $|y_{\text{opt}}|^{-m}$.

In the important special case $\alpha = -\infty, \beta = 0$, which occurs, e.g., for inverse fractional powers~\eqref{eq:inverse_fractional_powers}, the above formulas simplify and we obtain the pole $\xi = -\sqrt{\widetilde\lambda_{\max} \cdot \widetilde \lambda_{\min}}$ and the corresponding convergence rate
\begin{equation}\label{eq:convergence_rate_markov_simple_pole}
   \eta_m = %|y_{\text{opt}}|^{-m} =
   \left(\frac{\sqrt[4]{\widetilde\lambda_{\max} /\widetilde \lambda_{\min}}-1}{\sqrt[4]{\widetilde\lambda_{\max} /\widetilde \lambda_{\min}}+1}\right)^m.
\end{equation}

Let us note that, asymptotically, the convergence rate~\eqref{eq:convergence_rate_markov_simple_pole} is also attained when alternatingly choosing the poles $0$ and $\infty$, i.e., when using \emph{extended Krylov subspaces}~\cite{BeckermannReichel2009,KnizhnermanSimoncini2010}.

\begin{figure}
\begin{center}
\pgfplotsset{height=0.45\linewidth,width=0.95\linewidth,compat=1.10,every axis/.append style={legend style={/tikz/every even column/.append style={column sep=6pt}}}}

\noindent%
\begin{tikzpicture}[scale=1]%
    \begin{semilogyaxis}[legend style={at={(1,1.15)}}, 
   	anchor= north east, legend columns=2, xmin=0, xmax=200,grid=major, 
   	xlabel={$m$}, ylabel={Error norm}]

\addplot[color=NavyBlue,very thick]
table [x ={x},y ={err}] {figs/experiment1.dat};\addlegendentry{Exact error norm}
 
  \addplot[color=BurntOrange,very thick,dashed]
table [x ={x},y ={est}]{figs/experiment1.dat} node [pos=0,left] {}; \addlegendentry{Estimated convergence rate}
    \end{semilogyaxis}
\end{tikzpicture}
\end{center}
\caption{Convergence of $\|E_m(f)\|$ for Algorithm~\ref{alg:rationalkrylovupdate} with a single, repeated, asymptotically optimal pole, and estimated convergence rate~\eqref{eq:convergence_rate_markov_simple_pole} for approximating $(A+\vbv\vbv^\ast)^{-1/2}-A^{-1/2}$ with $A,\vbv$ as in Example~\ref{example:1}.}
\label{fig:invsqrt_simple_pole}
\end{figure}
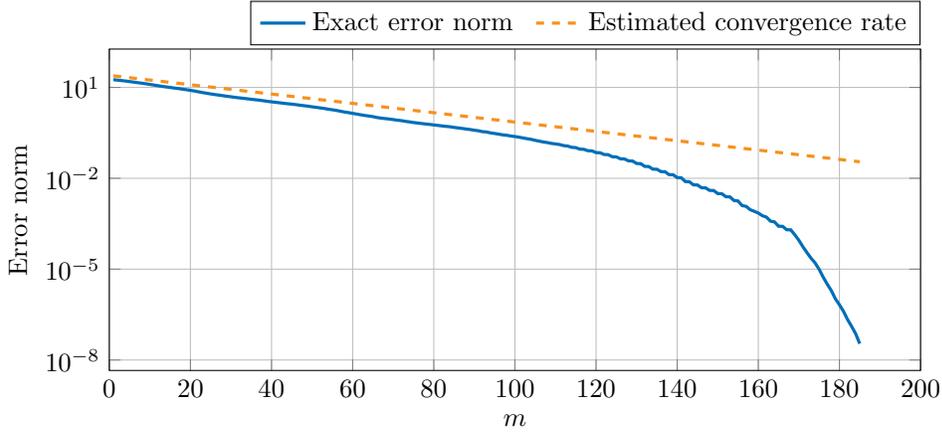

\begin{example} \label{example:1}
	We illustrate the above results by a simple numerical experiment, using a diagonal matrix $A \in \C^{200 \times 200}$ with logarithmically spaced eigenvalues in the interval $[10^{-3},10^3]$ and $D = \vbv\vbv^\ast$ where $\vbv$ is a random vector with $\|\vbv\| = 100$. This leads to $\widetilde\lambda_{max} \approx 1.0078 \cdot 10^4$, and thus $\widetilde\lambda_{max}/\widetilde\lambda_{min}) \approx 1.0078\cdot 10^7$. Figure~\ref{fig:invsqrt_simple_pole} displays the convergence of Algorithm~\ref{alg:rationalkrylovupdate} with all poles equal to $-\sqrt{\widetilde\lambda_{max}\cdot\widetilde\lambda_{min}}$ for approximating $(A+\vbv\vbv^\ast)^{-1/2}-A^{-1/2}$. In the initial phase, the error reduces linearly and the convergence rate of the method is predicted quite accurately by~\eqref{eq:convergence_rate_markov_simple_pole}. The superlinear convergence phase starting around iteration $120$ can of course not be captured by~\eqref{eq:convergence_rate_markov_simple_pole}. \hfill$\diamond$
\end{example}

We now turn to rational approximations using several different poles. In~\cite[Section 6.2]{BeckermannReichel2009}, quasi-optimal poles are constructed that admit closed formulas in terms of Jacobi elliptic functions. Using these poles, 
\begin{equation}\label{eq:rational_approximation_quasioptimal}
\eta_m \leq %\|f-r_m\|_\bbE \leq \frac{2\|f\|_\bbE}{|\phi(\beta)|}
2 \exp\left(-m \frac{\pi^2}{\log(16\widetilde \lambda_{\max} / \widetilde \lambda_{\min})}\right).
\end{equation}
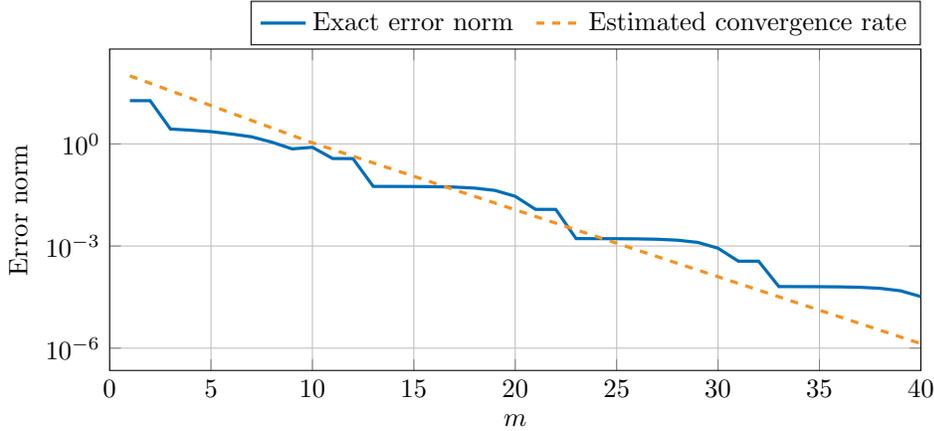
\begin{figure}
\begin{center}
\pgfplotsset{height=0.45\linewidth,width=0.95\linewidth,compat=1.10,every axis/.append style={legend style={/tikz/every even column/.append style={column sep=6pt}}}}

\noindent%
\begin{tikzpicture}[scale=1]%
    \begin{semilogyaxis}[legend style={at={(1,1.15)}}, 
   	anchor= north east, legend columns=2, xmin=0, xmax=40,grid=major, 
   	xlabel={$m$}, ylabel={Error norm}]

\addplot[color=NavyBlue,very thick]
table [x ={x},y ={err}] {figs/experiment2.dat};\addlegendentry{Exact error norm}
 
  \addplot[color=BurntOrange,very thick,dashed]
table [x ={x},y ={est}]{figs/experiment2.dat} node [pos=0,left] {}; \addlegendentry{Estimated convergence rate}
    \end{semilogyaxis}
\end{tikzpicture}
\end{center}
\caption{
Convergence of $\|E_m(f)\|$ for Algorithm~\ref{alg:rationalkrylovupdate} with 10 quasi-optimal, cyclically repeated poles, and estimated convergence rate~\eqref{eq:rational_approximation_quasioptimal_cyclic} for approximating $(A+\vbv\vbv^\ast)^{-1/2}-A^{-1/2}$ with $A,\vbv$ as in Example~\ref{example:1}.}
\label{fig:invsqrt_quasioptimal}
\end{figure}%
	Thus, the rate of convergence now depends on the \emph{logarithm} of the ratio $\widetilde \lambda_{\max}/\widetilde \lambda_{\min}$ instead of the fourth root. The corresponding poles are mutually distinct and, in turn, the rational Arnoldi method requires to compute a new Cholesky decomposition in each of the $m$ iterations. As already mentioned in Section~\ref{sec:rational}, it is preferable in practice to use a smaller number of poles and repeat them (typically cyclically) in order to limit the number of matrix factorizations that need to be computed. When using $\widetilde{m}$ quasi-optimal poles and repeating each of them $k$ times, the error bound~\eqref{eq:rational_approximation_quasioptimal} changes to
\begin{equation}\label{eq:rational_approximation_quasioptimal_cyclic}
    \eta_{m} \leq 2^k
    %\|f-r_{k m}\|_\bbE  \leq \frac{2^k \|f\|_\bbE}{|\phi(\beta)|}
	\exp\left(-k\widetilde{m} \frac{\pi^2}{\log(16\widetilde \lambda_{\max} / \widetilde \lambda_{\min})}\right),\quad m = k \widetilde{m}.
\end{equation}
Thus, compared to using $m$ (mutually distinct) quasi-optimal poles, the error bound worsens by a factor $2^{k-1}$.

We repeat the experiment from Example~\ref{example:1}, now using ten cyclically repeated, quasi-optimal poles in Leja ordering~\cite{Reichel1990}. Figure~\ref{fig:invsqrt_quasioptimal} displays the resulting convergence. The overall convergence rate is again predicted quite accurately, although the actual convergence curve shows a staircase-like behavior (which is typical for rational Krylov methods with poles in Leja ordering).
%\end{example}

Other, practically relevant functions like the \emph{matrix square root} are obtained as slight modifications of Markov functions.

\begin{example}\label{ex:square_root} 
Let us consider functions of the form
\begin{equation}\label{eq:markov_modification}
f(z) = z\widehat{f}(z),
\end{equation}
where $\widehat{f}$ is a Markov function~\eqref{eq:markov}. 
This includes the square root $z^{1/2} = zz^{-1/2}$ as well as the logarithm $\log(1+z) = z\frac{\log(1+z)}{z}$. The following simple trick allows us to apply 
Theorem~\ref{the:estimate_rational_Markov} to this setting. Fixing the pole $\xi_m=\infty$, which gives $q_m=q_{m-1} \in \Pi_{m-1}$, and setting $p_1(z)=z$ we obtain
\begin{eqnarray*} 
    \min_{r \in \Pi_{m}/q_{m}} \| f - r \|_{L^\infty(\bbE)}
    &\leq &
    \min_{r \in \Pi_{m-1}/q_{m}} \| \widehat f - r \|_{L^\infty(\bbE)} \,
    \| p_1 \|_{L^\infty(\bbE)}
    \\&= &
    \min_{r \in \Pi_{m-1}/q_{m-1}} \| \widehat f - r \|_{L^\infty(\bbE)} \,
    \| p_1 \|_{L^\infty(\bbE)} .
\end{eqnarray*}
That is, besides the additional factor $\| p_1 \|_{L^\infty(\bbE)}$, we obtain an upper bound for $\tildeE(f)$ by combining Theorem~\ref{the:convergence_rational} for $m,f$ with Theorem~\ref{the:estimate_rational_Markov} for $m-1,\widehat f$. A similar technique has been used in~\cite{FrommerGuettelSchweitzer2014} in the context of convergence theory for restarted (polynomial) Krylov methods for $f(A)\vbv$ when $A$ is Hermitian positive definite. In that situation, $\| p_1 \|_{L^\infty(\bbE)} = \lmax$.
\hfill$\diamond$
\end{example}

%\begin{example}\label{ex:exponential} Finally, let us briefly discuss the case of the exponential, $f(z) = \exp(-z)$, for a Hermitian positive definite matrix $A$. In this case, results from~\cite{MN04,vdEH06} imply that a well-chosen single pole $\xi < 0$ repeated $m$ times already yields good convergence. Strategies for choosing more than one pole (adaptively) are surveyed in~\cite[Sec. 4.2]{Guettel2013}. ??? Include Andersson 1981 specific rate ???
%\end{example}

\subsubsection{Convergence analysis for Markov functions in the non-Hermitian case}

We now turn to the more difficult task of analyzing the convergence for general $A$, $D = \vb \vc^\star$, in terms of a convex and compact set $\mathbb E$ containing both numerical ranges $W(A)$ and $W(A+\vb\vc^*)$, and $f$ being analytic in $\mathbb E$. In principle, Theorem~\ref{thm:poly} also holds for rational Krylov subspaces, by replacing $p$ with the derivative of a function in $\Pi_m / q_m $. However, due to the special form of such a derivative, the resulting optimization problem appears to be too exotic to be of assistance in getting practical convergence bounds. Therefore, inspired by \cite[\S 5.1]{BeckermannKressnerSchweitzer2017}, we consider the shifted (block) linear systems
\begin{equation}\label{eq:shifted_systems}
(zI - A)\vx(z) = \vb \text{ and } (zI - A -  \vb\vc^\ast)^\ast \vy(z) = \vc,
\end{equation}
together with the rational block FOM approximations for~\eqref{eq:shifted_systems}, given by
\begin{eqnarray*}
\vx_m(z) &:=& U_m (zI-G_m)^{-1}U_m^\ast \vb, \nonumber\\
 \vy_m(z) &:=& V_m(\bar z I - H_m - V_m^\ast \vc  \vb^\ast V_m)^{-1}V_m^\ast \vc.
\end{eqnarray*}
The following result links these quantities to the approximation error of low-rank updates.
\begin{lemma}\label{lem:contour}
   With $\Gamma$ a contour surrounding $\mathbb E$ once and sufficiently close to $\mathbb E$, the error defined in~\eqref{eq:error} satisfies
   $$
         \tildeE(f) = \frac{1}{2\pi i} \int_\Gamma f(z) \big(\vx(z)\vy(z)^\ast -  \vx_m(z)\vy_m(z)^\ast\big) \d z.
   $$
\end{lemma}
\begin{proof}
This result has been derived in~\cite[\S 5.1]{BeckermannKressnerSchweitzer2017} in the context of polynomial Krylov subspaces, but it is straightforward to verify that the derivations are valid for general choices of subspaces.
%     Using the Cauchy integral formula, we have
%     \begin{eqnarray*}
%       f(A + \vb\vc^\ast) - f(A)
%       &=& \frac{1}{2 \pi i}\int_\Gamma f(z) \bigl( (zI-A-\vb \vc^\ast)^{-1} - (zI-A)^{-1} \bigr) \d z =
%       \\&=& \frac{1}{2 \pi i}\int_\Gamma f(z) \vx(z)\vy(z)^\ast \d z
%     \end{eqnarray*}
%     and, similarly,
%     \begin{eqnarray*}
%         U_m X_m(f) V_m^* &=&
%          \frac{1}{2 \pi i}\int_\Gamma f(z)
%         U_m \left( \left( zI-  \begin{bmatrix}
% G_m &  (U_m^\ast \vb) (V_m^\ast \vc)^\ast  \\
% 0 & H_m^\ast + (V_m^\ast \vb) (V_m^\ast \vc)^\ast
% \end{bmatrix} \right)^{-1} \right)_{1,2} V_m^\ast \, \d z
%       \\&=& \frac{1}{2 \pi i}\int_\Gamma f(z) \vx_m(z)\vy_m(z)^\ast \d z.
%     \end{eqnarray*}
%     Inserting both expressions into~\eqref{eq:error} establishes the result.
\end{proof}

Lemma~\ref{lem:contour} shows that $\|\tildeE(f)\|$ is small if the rational FOM approximation errors $\vx(z)-  \vx_m(z)$ and $\vy(z)-  \vy_m(z)$ are small, uniformly for $z\in \Gamma$. The analysis is complicated by this dependence on $\Gamma$. Therefore, in what follows we will only consider the particular case \eqref{eq:markov} of a Markov function $f$, which allows us to switch from $\Gamma$ to the interval $[\alpha,\beta]$.

\begin{theorem}\label{the:unsymmetric_Markov}
   Let $\mathbb E$ be a convex and compact set, symmetric with respect to the real axis, and containing both numerical ranges $W(A)$ and $W(A+\vb\vc^*)$. Let $\omega$ and $\eta_m$ be defined as in Theorem~\ref{the:estimate_rational_Markov}.
   Then for a Markov function $f$ satisfying~\eqref{eq_assumtion_markov}, the error~\eqref{eq:error} returned by Algorithm~\ref{alg:rationalkrylovupdate} satisfies
   $$
        \| \tildeE(f) \| \leq 8 \, |f'(\omega)|\,
        \frac{\eta_m}{1-\eta_m} \, \| \vb \|\, \| \vc \|.
   $$
\end{theorem}
\begin{proof}
In the same way as in the proof of~\cite[Theorem 5.7]{BeckermannKressnerSchweitzer2017}, we obtain from Lemma~\ref{lem:contour} and the Fubini theorem the bound
% Since all expressions
% $\vx(z),\vy(z)^*,\vx_m(z),\vy_m(z)^*$ are analytic outside $\bbE$ and decay like $1/z$ at $\infty$, we get from Lemma~\ref{lem:contour} by exchanging integration and using the Cauchy residual theorem
% \begin{eqnarray*}
%      \tildeE(f)&=&\frac{1}{2\pi i} \int_{\Gamma} \int_{\alpha}^\beta \frac{\d\mu(t)}{z-t} \big( \vx(z)\vy(z)^*- \vx_m(z)\vy_m(z)^\ast \big) \d z
%      \\&=& -\int_{\alpha}^\beta \big( \vx(t)\vy(t)^*-\vx_m(t)\vy_m(t)^\ast \big) \d\mu(t)
%      \\&=& -\int_{\alpha}^\beta \big( \vx(t)(\vy(t)-\vy_m(t))^\ast + (\vx(t) - \vx_m(t))\vy_m(t)^\ast \big) \d\mu(t) .
% \end{eqnarray*}
% By taking norms, we obtain
\begin{equation}
\|\tildeE(f)\| \leq \int_\alpha^\beta ( \|\vx(t)\|\|\vy(t)-\vy_m(t)\| + \|\vy_m(t)\| \|\vx(t) - \vx_m(t))\|) \dmu \label{eq:Em_norm}.
\end{equation}
We have \begin{equation} \label{eq:boundshiftedinverse}
         \|(tI-A)^{-1}\| \leq \frac{1}{\mbox{dist}(t,W(A))} \le  \frac{1}{\mbox{dist}(t,\bbE)} \le \frac{1}{\omega-t},
        \end{equation}
  where the last inequality follows for all $t\in [\alpha,\beta]$ from condition~\eqref{eq_assumtion_markov}.
 Analogously,
\[
 \|(t I - H_m - V_m^\ast \vc  \vb^\ast V_m)^{-1}\| 
 %\frac{1}{\mbox{dist}(t,W(H_m + V_m^\ast \vc  \vb^\ast V_m))}
 \leq \frac{1}{\mbox{dist}(t,W(A+\vb\vc^\ast))} \le \frac{1}{\omega-t}.
\]
In particular, these bounds imply
$$
    \frac{\|\vx(t)\|}{\| \vb \|} \leq \frac{1}{\mbox{dist}(t,\bbE)}
    = \frac{1}{\omega-t} , \quad
    \frac{\|\vy_m(t)\|}{\| \vc \|} \leq \frac{1}{\omega-t}
$$
for all $t\in [\alpha,\beta]$.
We claim that, for $t\in [\alpha,\beta]$,
\begin{equation} \label{eq:claim}
    \frac{\|\vx(t) - \vx_m(t) \|}{\| \vb \|} \leq
    \frac{4}{\omega-t}\cdot \frac{\eta_m}{1-\eta_m} , \quad
    \frac{\|\vy(t) - \vy_m(t) \|}{\| \vc \|} \leq
    \frac{4}{\omega-t}\cdot \frac{\eta_m}{1-\eta_m} .
\end{equation}
Inserting these bounds into \eqref{eq:Em_norm} leads to
$$
     \| \tildeE(f) \| \leq 8 \| \vb\| \| \vc\| \frac{\eta_m}{1-\eta_m}\,
     \int \frac{d\mu(t)}{(\omega-t)^2},
$$
with the integral being equal to $|f'(\omega)| = \| f' \|_{L^\infty(\mathbb E)}$. Hence, we arrive at the assertion of the theorem.

It remains to show the first inequality of \eqref{eq:claim}, the proof of the second is entirely analogous. Theorem~3.4 in~\cite{Beckermann2011} establishes the existence of a rational function $R \in \Pi_m/q_m$ depending only on $q_m$ and $\bbE$ such that
\begin{equation} \label{eq:propertiesofR}
\begin{array}{l}
    \| R(\widetilde A) \| \leq  2 \mbox{~for all square matrices $\widetilde A$ with $W(\widetilde A)\subset \bbE$},
   \\
   | R(z) |\leq 2 \mbox{~for all $z\in \bbE$},
   \\ | R(t) |\geq |B_m(\phi(t))| - 1 \mbox{~for all $t\not\in \bbE$}.
  \end{array}
\end{equation}
Let $t\in [\alpha,\beta]$ be fixed, and consider the rational function 
$$
      z \mapsto r_t(z) = \frac{1}{z-t} - \frac{1}{z-t} \cdot \frac{R(z)}{R(t)} .
$$
Since $r_t\in \Pi_{m-1}/q_m$, the exactness property of Lemma~\ref{lemma:exactrational} allows to conclude that
$r_t(A)\vb = U_m r_t(G_m) U_m^\ast \vb$, and thus
\begin{equation} \label{xintermsofR}
     \vx(t) - \vx_m(t)
     = (tI-A)^{-1} \frac{R(A)}{R(t)} \vb - U_m(tI-G_m)^{-1} \frac{R(G_m)}{R(t)} U_m^\ast \vb.
\end{equation}
Using the properties of $R$ from~\eqref{eq:propertiesofR} and the bound~\eqref{eq:boundshiftedinverse}, we have
$$
     \left\| (tI-A)^{-1} \frac{R(A)}{R(t)} \vb \right\| \leq \frac{\| \vb\|}{\omega-t} \frac{\| R(A) \|}{|R(t)|} \leq
     \frac{2 \| \vb\|}{\omega-t} \frac{\eta_m}{1-\eta_m}
$$
and the same upper bound if one replaces $\vb$ and $A$ by $U_m^\ast \vb$ and $G_m$, respectively. Inserting these bounds into~\eqref{xintermsofR} shows the claim \eqref{eq:claim} and completes the proof.
\end{proof}

% \begin{remark}\label{rem:subdiagonal}
% \textcolor{red}{??? I think we should clarify the purpose of this remark. Is the point that we rediscover the result of Theorem 4.2? Up to a constant; how much can the constants differ for large $m$? Is this proof technique similar or very different from [8]? In which sense does this improve the result of Theorem~\ref{the:estimate_rational_Markov}?}
%    With the notation of Theorem~\ref{the:estimate_rational_Markov}, one may also show that
%    $$
%         \min_{r \in \Pi_{m-1}/q_{m}} \| f - r \|_{L^\infty(\bbE)}
%         \leq 2 \, \| f \|_{L^\infty(\mathbb E)} \frac{\eta_m}{1-\eta_m}.
%    $$
%    With the notations of the previous proof, it is sufficient to examine the candidate $r(z)=\int r_t(z) d\mu(t)\in \Pi_{m-1}/q_m$, since, for $z\in \bbE$, 
%    $$
%         |f(z)-r(z)| = \left| \int_\alpha^\beta \frac{R(z)}{R(t)}d\mu(t) \right|
%         \leq \frac{2\eta_m}{1-\eta_m} \int \frac{\d\mu(t)}{\gamma-t}.
%    $$
%    ~\hfill$\diamond$
% \end{remark}

\begin{remark}\label{rmk:polynomial_convergence} \em
For polynomial Krylov subspaces, $\xi_1= \cdots =\xi_m=\infty$. In turn, $B_m(\phi(x))=\phi(x)^m$ and $\eta_m=1/|\phi(\beta)|^m$. Thus, up to the factor $1/(1-\eta_m)$, our Theorem~\ref{the:unsymmetric_Markov} reduces to  \cite[Theorem~5.7]{BeckermannKressnerSchweitzer2017}. We mention in passing that this factor can be removed, using the techniques of \cite[Lemma~5.1]{BeckermannKressnerSchweitzer2017}, if at least two of the poles $\xi_1,...,\xi_m$ are infinite.
We should also mention that, once a suitable set $\bbE$ with more explicit conformal map $\phi$ (as for instance an ellipse or a teardrop set) is found, we may use some of the estimates for $\eta_m$ in terms of $\phi,\alpha,\beta$ as stated in \S \ref{sec_markov}.
\end{remark}

% SHALL WE ADD OTHER EXAMPLES? 
% \textcolor{red}{??? YES, it would be good to give at least one practically relevant example that does not involve the real line / a real interval. Otherwise we have no specific result in this section going significantly beyond the Hermitian case. What about inverse square root for a wedge set in the left-half plane???}

%\input{experiments.tex}

\section{The matrix sign function}\label{sec:sign}

When the numerical range of $A$ or $A+\vb\vc^*$ contains a singularity of $f$, none of the convergence results from Section~\ref{sec:convergence} applies. For the matrix sign function, a notorious example for this situation, we discuss a potential remedy.

Letting
\[
\sign: \C\setminus {\mathrm{i} \mathbb R} \to \C,\qquad \sign(z) = \begin{cases}
-1 & \Re(z) < 0, \\
1 & \Re(z) > 0,
\end{cases}
\]
where $\Re(z)$ denotes the real part of $z$, the matrix sign function $\sign(A)$ is defined whenever $A$ has no purely imaginary eigenvalue. This function plays an important role in, e.g., linear-quadratic optimal control~\cite{Roberts1980}, quantum chromodynamics~\cite{BlochEtAl2007,VanDenEshofFrommerLippertSchillingVanDerVorst2002}, and eigenvalue solvers~\cite{BeylkinCoultMohlenkamp1999,NakaHigh2013}.

\subsection{Low-rank updates}

Except for trivial situations ($\sign(A) = \pm I$), the sign function is usually \emph{not} defined on the numerical range of $W(A)$, which poses a severe problem for Krylov subspace techniques, not only in theory but also in practice. In the context of approximating $\sign(A)\vbv$, Krylov subspace methods have been observed to exhibit slow, irregular or erratic convergence~\cite{VanDenEshofFrommerLippertSchillingVanDerVorst2002}.
%We have made a similar observation in the context of our algorithm; see Example~\ref{ex:sign_algorithm} below.
As a remedy, it has been proposed to exploit the relation
\begin{equation}\label{eq:sign_function_inverse_square_root}
\sign(A) = (A^2)^{-1/2}A
\end{equation}
and approximate $\sign(A)\vbv$ in the Krylov space $\spK_m(A^2,A\vbv)$; see, e.g.,~\cite{Borici1999,VanDenEshofFrommerLippertSchillingVanDerVorst2002}. For an invertible Hermitian matrix $A$ the advantage of~\eqref{eq:sign_function_inverse_square_root} is obviously that the numerical range of $A^2$ does not contain a singularity of the inverse square root.% Each iteration of a Krylov method for $\spK_m(A^2,A\vb)$ requires two matrix-vector products per iteration, but it often converges in roughly half the number of iterations as a Krylov method for $\spK_m(A,\vb)$, so that both approaches have roughly the same computational cost, see, e.g.,~\cite[Section 4.4]{VanDenEshofFrommerLippertSchillingVanDerVorst2002} and also Example~\ref{ex:sign_algorithm} below.

In the following, we will discuss an approach based on~\eqref{eq:sign_function_inverse_square_root} for approximating low-rank updates~\eqref{eq:low_rank_update} of the matrix sign function. Because~\eqref{eq:sign_function_inverse_square_root} offers a clear advantage only for the Hermitian case, we now assume that $A = A^\ast$ and $D = \vb J \vb^\ast$ with $J = J^\ast$. Let us, however, mention that the construction readily extends to the non-Hermitian case.

Using~\eqref{eq:sign_function_inverse_square_root}, it follows that
\begin{align}
	&\sign(A+D) - \sign(A)  =  (A+D)((A+D)^2)^{-1/2} - A(A^2)^{-1/2} \nonumber \\
=& (A+D)\big((A^2+\widetilde D)^{-1/2} - (A^2)^{-1/2}\big) + \vb J\vb^\ast(A^2)^{-1/2} \label{eq:signformulation}
\end{align}
with $\widetilde D := A\vb J\vb^\ast+\vb J\vb^\ast(A+\vb J\vb^\ast)$.
A rank-$\ell$ update of the sign function is thus performed by computing a rank-$2\ell$ update of $(A^2)^{-1/2}$ and the action of $(A^2)^{-1/2}$ on $\vb$. Because the range and co-range of $\widetilde D$ are contained in the span of $[\vb,A\vb]$, it is natural to choose the rational Krylov subspace 
\begin{equation}\label{eq:block_krylov_space_sign}
\calU_m := q_m(A^2)^{-1}\spK_m(A^2,[\vb,A\vb])
\end{equation}
with suitably chosen poles $\xi_1, \ldots,\xi_m$ for approximating the rank-$2\ell$ update. To approximate the second term in~\eqref{eq:signformulation}, we utilize the usual 
block Krylov approximation
$$(A^2)^{-1/2}\vb \approx U_m G_m^{-1/2}U_m^\ast\vb$$
for an orthonormal basis $U_m$ of  $\calU_m$.
Algorithm~\ref{alg:sign} summarizes the described approach for approximating~\eqref{eq:signformulation}.

\begin{algorithm}
\caption{Rational block Krylov subspace approximation of sign matrix function update~\eqref{eq:signformulation} for Hermitian $A,D$}
\label{alg:sign}
\begin{algorithmic}[1]
\State Choose poles $\xi_1,\ldots,\xi_m \in \C \cup \{\infty\}$ closed under complex conjugation.
\State Perform $m$ steps of Algorithm~\ref{alg:rational_arnoldi} to compute an orthonormal basis $U_m$ of $\calU_m = q_m(A^2)^{-1}\spK_m(A^2,[\vb,A\vb])$
and set $G_m = U_m^\ast A^2 U_m$.

\State Compute $F_m \!=\! \setlength{\arraycolsep}{1pt}\left(\begin{bmatrix}
G_m & U_m^\ast\big(A\vb J\vb^\ast + \vb J\vb(A+\vb J\vb^\ast)\big)U_m \\ 0 & G_m + U_m^\ast\big(A\vb J\vb^\ast + \vb J\vb(A+\vb J\vb^\ast)\big)U_m\end{bmatrix}\right)^{-1/2}.$
\State Set $X_m(z^{-1/2}) = F_m(1\!:\!2m,2m+1\!:\!4m)$.
\State Compute $\vf_m = U_m G_m^{-1/2}U_m^\ast\vb$.
\State Return $(A+\vb J\vb^\ast)(U_m X_m(z^{-1/2}) U_m^\ast) + \vb J\vf_m^\ast$.
\end{algorithmic}
\end{algorithm}

\begin{remark}\label{rmk:krylov_space_sign} \em
The rational Krylov space~\eqref{eq:block_krylov_space_sign} used in Algorithm~\ref{alg:sign} has a very specific structure, and its polynomial part is actually identical to an ordinary block Krylov space of order $2m$ for $A$. Precisely
\begin{equation}\label{eq:krylov_space_sign_rewritten}
	\calU_m = q_m(A^2)^{-1}\spK_m(A^2,[\vb,A\vb]) = q_{m}(A^2)^{-1}\spK_{2m}(A,\vb).
\end{equation}
This is different from the situation arising when approximating $\sign(A)\vbv$, where the polynomial part of the subspace corresponds only to odd powers of $A$. When $\vb$ is a vector, this observation could in principle be used to implement Algorithm~\ref{alg:sign} such that it avoids block arithmetic.
\end{remark}

The convergence of Algorithm~\ref{alg:sign} can be analyzed by combining the results from Section~\ref{sec:convergence} with known convergence results for Krylov subspace methods.

\begin{theorem}\label{the:sign_convergence}
Let $A$ and $D = \vb J \vb^\ast$ be Hermitian such that $A$ and $A+D$ are invertible. Then the error of the approximation returned by Algorithm~\ref{alg:sign} satisfies
\begin{align}
 &\|\sign(A+D) - \sign(A) - (A+D)(U_m X_m(z^{-1/2}) U_m^\ast) + \vb J\vf_m^\ast\|  \label{eq:tobound} \\
\leq& (4 \|A+D\| +2 \|\vb J\|\,\|\vb\|)  \min_{r \in \Pi_{m}/q_{m}} \|f - r\|_\bbE, \nonumber 
\end{align}
where $\bbE = \big[ \min\{\lmin(A^2),\lmin((A+D)^2)\},  \max\{\lmax(A^2),\lmax((A+D)^2)\} \big]$ and $f(z) = z^{-1/2}$.
\end{theorem}
\begin{proof}
Using~\eqref{eq:signformulation} and setting $M = (A^2+\widetilde D)^{-1/2} - (A^2)^{-1/2}$, it follows that~\eqref{eq:tobound} is bounded by
\begin{align*}
& \| (A+D)\big(M - U_m X_m(f) U_m^\ast\big) + \vb J \big( \vb^\ast(A^2)^{-1/2} - \vf_m^\ast \big) \| \\
\le & \|A+D\|\,\|M - U_m X_m(f) U_m^\ast\| + \|\vb J\|\, \| \vb^\ast(A^2)^{-1/2} - \vf_m^\ast \|
\end{align*}
Using Theorem~\ref{the:convergence_rational}, the first term is bounded via
\begin{equation}\label{eq:sign_proof_estimate1}
\|M - U_mX_m(f)U_m^\ast\| \leq 4  \min_{r \in \Pi_{m}/q_{m}} \|f - r\|_\bbE.
\end{equation}
For the second term, we can estimate
\begin{equation}\label{eq:sign_proof_estimate2}
\|(A^2)^{-1/2}\vb - \vf_m^\ast\| \leq 2 \|\vb\| \min_{r \in \Pi_{m}/q_{m}} \|f - r\|_{\widetilde{\bbE}} \leq 2 \|\vb\| \min_{r \in \Pi_{m}/q_{m}} \|f - r\|_{\bbE}
\end{equation}
	with $\widetilde{\bbE} = [\lmin(A^2), \lmax(A^2)] \subseteq \bbE$. For the case that $\vb$ is a vector,~\eqref{eq:sign_proof_estimate2} is shown in \cite[Theorem 4.10]{Guettel2010}, see also the proof of~\cite[Theorem~5.2]{BeckermannReichel2009}, and the proof of this result carries over to the block case (and the non-standard rational Krylov space that we are using) completely analogously, using the exactness property from Lemma~\ref{lemma:exactrational} as a basis. Further note that the estimate~\eqref{eq:sign_proof_estimate2} is actually valid for the smaller subspace $q_m(A^2)^{-1}\spK_m(A^2,\vb) \subseteq q_m(A^2)^{-1}\spK_m(A^2,[\vb,A\vb])$.
%, but it is difficult to quantify the speed-up of convergence due to the augmentation with $A\vb$.
Combining~\eqref{eq:sign_proof_estimate1} and~\eqref{eq:sign_proof_estimate2} gives the desired result.
\end{proof}

As $f(z) = z^{-1/2}$ is a Markov function, we can, e.g., apply Theorem~\ref{the:estimate_rational_Markov} to obtain bounds for $\|f - r\|_\bbE$ in Theorem~\ref{the:sign_convergence}. 

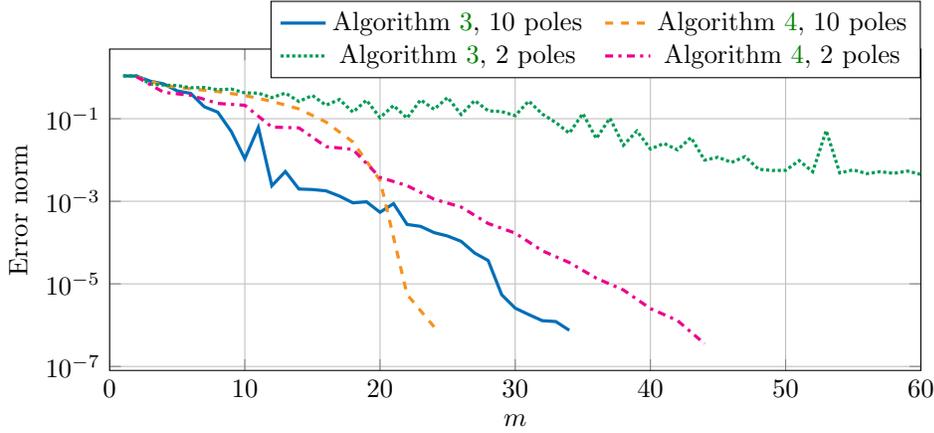
\begin{figure}
\begin{center}
\pgfplotsset{height=0.45\linewidth,width=0.95\linewidth,compat=1.10,every axis/.append style={legend style={/tikz/every even column/.append style={column sep=6pt}}}}

\noindent%
\begin{tikzpicture}[scale=1]%
    \begin{semilogyaxis}[legend style={at={(1,1.15)}}, 
   	anchor= north east, legend columns=2, xmin=0, xmax=60,grid=major, 
   	xlabel={$m$}, ylabel={Error norm}]

\addplot[color=NavyBlue,very thick]
	    table [x ={x},y ={err1}] {figs/experiment3.dat};\addlegendentry{Algorithm~\ref{alg:rationalkrylovupdate}, $10$ poles}
 
  \addplot[color=BurntOrange,very thick,dashed]
	    table [x ={x},y ={err2}]{figs/experiment3.dat} node [pos=0,left] {}; \addlegendentry{Algorithm~\ref{alg:sign}, $10$ poles}

\addplot[color=ForestGreen,very thick,densely dotted]
	    table [x ={x},y ={err3}] {figs/experiment3.dat};\addlegendentry{Algorithm~\ref{alg:rationalkrylovupdate}, $2$ poles}
 
  \addplot[color=Magenta,very thick,dash dot]
	    table [x ={x},y ={err4}]{figs/experiment3.dat} node [pos=0,left] {}; \addlegendentry{Algorithm~\ref{alg:sign}, $2$ poles}
    \end{semilogyaxis}
\end{tikzpicture}
\end{center}
\caption{Convergence curves of Algorithm~\ref{alg:rationalkrylovupdate} and Algorithm~\ref{alg:sign} using the poles of a Zolotarev approximation of degree 2 or 10 for approximating $\sign(A+\vbv\vbv^\ast)-\sign(A)$, where $\spec(A) \subseteq [-1,-10^{-2}] \cup [10^{-2},1],\ \|\vbv\| = 1$.}
\label{fig:diagonal_sign}
\end{figure}

\begin{example}\label{ex:sign_algorithm}
Consider the diagonal, indefinite matrix $A \in \C^{200 \times 200}$ with $100$ linearly spaced eigenvalues in each of the intervals $[-1,-10^{-2}]$ and $[10^{-2},1]$. Let $\vbv \in \C^{200}$ be a random vector of unit norm. We compare Algorithm~\ref{alg:sign} to the straight-forward application of Algorithm~\ref{alg:rationalkrylovupdate} to perform the update $\sign(A+\vbv\vbv^\ast)-\sign(A)$. We use the poles of the Zolotarev approximation of degree 2 and 10 for the inverse square root in Algorithm~\ref{alg:sign} and the poles of the corresponding Zolotarev approximation of the sign function in Algorithm~\ref{alg:rationalkrylovupdate}; see~\cite{PetrushevPopov1988,Zolotarev1877}. Again, the poles are in Leja ordering and cyclically repeated. The resulting convergence curves are depicted in Figure~\ref{fig:diagonal_sign}.
%, where we aim for an accuracy of $10^{-6}$.
As expected, the convergence curve of Algorithm~\ref{alg:sign} is much smoother than that of Algorithm~\ref{alg:rationalkrylovupdate}. In addition, the subspace dimension required to reach the target accuracy $10^{-6}$ by Algorithm~\ref{alg:sign} is smaller: When using 10 different poles, it needs 24 vs 34 iterations, i.e., a reduction of about 30\%. For only $2$ different poles, the difference becomes a lot more pronounced, and Algorithm~4 requires 44 iterations, while Algorithm~3 fails to converge in a reasonable number of iterations.

Concerning the computation cost of the algorithms, several things have to be taken into account: On the one hand, the number of nonzeros in $A^2$ is typically larger than in $A$, which leads to higher expenses when factoring $A^2 + \xi_i I$. On the other hand, the poles of the Zolotarev approximation for the sign function are complex, so that Algorithm~\ref{alg:rationalkrylovupdate} requires complex arithmetic even though $A$ and $\vbv$ are real (note however, that only half the number of Cholesky factorizations needs to be computed, as the Zolotarev shifts come in complex conjugate pairs).
% DK: I have removed this discussion because I am not sure we want to advocate Algorithm 3.
%so that the operations performed in Algorithm~\ref{alg:sign} require more than twice the computation time of those in Algorithm~\ref{alg:rationalkrylovupdate}.
%Also note that in later iterations, the error norm in Algorithm~\ref{alg:rationalkrylovupdate} stagnates in every other iteration and it is therefore not advisable to use the error estimate~\eqref{eq:error_estimate_difference} with $d = 1$~\ref{alg:rationalkrylovupdate}. The stagnations are related to the occurrence of Ritz values very close to zero. Such a Ritz value near zero can only occur in two consecutive iterations if it is very close to an eigenvalue of $A$, so that stagnation for more than one iteration only occurs if $A$ is almost singular; see~\cite[Section 4.3]{VanDenEshofFrommerLippertSchillingVanDerVorst2002}. Consequently all values $d > 1$ can reasonably be used in the error estimator for the sign function.
\hfill$\diamond$
\end{example}

\subsection{Connection to Krylov subspace methods for linear matrix equations}

We conclude this work by pointing out a curious connection to Krylov subspace methods for 
the matrix Sylvester equation
\begin{equation} \label{eq:sylv}
 A_1 Z - Z A_2 + \vb_1 \vc_2^* = 0,
\end{equation}
with coefficients $A_1 \in \C^{n_1 \times n_1}$, $A_2 \in \C^{n_2 \times n_2}$ and $\vb_1 \in \C^{n_1 \times \ell}$, $\vc_2 \in \C^{n_2 \times \ell}$ such that $\ell \ll \min\{n_1,n_2\}$. We refer to~\cite{Simoncini2016} for an overview of applications and numerical algorithms for this and similar equations.

We assume that $W(A_1), W(-A_2)$ are contained in the open right-half plane, which implies that~\eqref{eq:sylv} has a unique solution $Z$. Moreover, it is well known that
\[
  \sign \left(\begin{bmatrix} A_1 & \vb_1 \vc_2^* \\ 0 & A_2 \end{bmatrix}\right) = \begin{bmatrix}
I_{n_1} & 2Z \\
0 & -I_{n_2}
\end{bmatrix}.
\]
In turn,
\begin{equation} \label{eq:signsylv}
\sign \left(\begin{bmatrix} A_1 & \vb_1 \vc_2^* \\ 0 & A_2 \end{bmatrix}\right) - \sign \left(\begin{bmatrix} A_1 & 0 \\ 0 & A_2 \end{bmatrix}\right)  = \begin{bmatrix}
0 & 2Z \\
0 & 0
\end{bmatrix}, 
\end{equation}
showing that the solution $Z$ of~\eqref{eq:sylv} can be obtained from a rank-$\ell$ update of the matrix sign function. After setting
\[
 A = \begin{bmatrix} A_1 & 0 \\ 0 & A_2 \end{bmatrix}, \quad \vb = \begin{bmatrix} \vb_1 \\ 0 \end{bmatrix}, \quad \vc = \begin{bmatrix} 0 \\ \vc_2 \end{bmatrix},
\]
the left-hand side of~\eqref{eq:signsylv} takes the familiar form $\sign(A+\vb\vc^*)-\sign(A)$.

As we will see below, the particular structure of the update implies that the squaring trick from the previous subsection is not needed for~\eqref{eq:signsylv}. Applying Algorithm~\ref{alg:rationalkrylovupdate} directly to~\eqref{eq:signsylv} involves the rational Krylov subspaces
\begin{eqnarray*}
 q_m(A)^{-1} \spK_m(A,\vb)  %&=&  q_m\left( \begin{bmatrix} A_1 & 0 \\ 0 & A_2 \end{bmatrix}\right)^{-1} \mathcal K_m\left( \begin{bmatrix} A_1 & 0 \\ 0 & A_2 \end{bmatrix}, \begin{bmatrix} \vb_1 \\ 0  \end{bmatrix}, \right)  \\
 &=& \left\{ \begin{bmatrix} u \\ 0 \end{bmatrix}: u\in q_m(A_1)^{-1} \spK_m(A_1,\vb_1) \right\}, \\
	 \bar{q}_m(A)^{-*} \spK_m(A^*,\vc) %&=& q_m\left( \begin{bmatrix} A_1 & 0 \\ 0 & A_2 \end{bmatrix}\right)^{-*} \mathcal K_m\left( \begin{bmatrix} A^*_1 & 0 \\ 0 & A^*_2 \end{bmatrix}, \begin{bmatrix} 0 \\ \vc_2  \end{bmatrix} \right) \\
 &=& \left\{ \begin{bmatrix} 0 \\ v \end{bmatrix}: v\in q_m(A_2)^{-*} \spK_m(A_2^*,\vc_2) \right\}.
\end{eqnarray*}
Thus, we obtain orthonormal bases $U_m = \begin{bmatrix} U_{1,m} \\ 0 \end{bmatrix}$, $V_m = \begin{bmatrix} 0 \\ V_{2,m} \end{bmatrix}$ by letting $U_{1,m}$ and $V_{2,m}$ contain orthonormal bases of $q_m(A_1)^{-1} \spK_m(A_1,\vb_1)$ and $q_m(A_2)^{-*}\spK_m(A_2^*,\vc_2)$, respectively. The compressions of $A$ and $A^*$ to these bases take the form
\[
G_m := U_m^* A U_m = U_{1,m}^* A_1 U_{1,m}, \quad 
                    H_m := V_m^* A^* V_m = V_{2,m}^* A^*_2 V_{2,m}.
\]
We recall that the matrix $X_m(\sign)$ in Algorithm~\ref{alg:rationalkrylovupdate} is extracted from the (1,2) block of the matrix~\eqref{eq:blockmatrix_krylov}. In the described setting, this matrix takes the form
\[
 \sign\left(
\begin{bmatrix}
\tildeG & U_{1,m}^* \vb_1 \vc_2^* V_{2,m} \\ 0 & H_m^\ast 
\end{bmatrix} \right) = \begin{bmatrix}
I & 2\tilde Z_m \\ 0 & -I
\end{bmatrix},
\]
where $\tilde Z_m$ satisfies the Sylvester equation $G_m \tilde Z_m - \tilde Z_m H_m^* + U_{1,m}^* \vb_1 \vc_2^* U_{2,m} = 0$, which has a unique solution because  of
$W(G_m) \subset W(A_1)$, $W(H^*_{m}) \subset W(A^*_2)$.

In summary, Algorithm~\ref{alg:rationalkrylovupdate} applied to~\eqref{eq:signsylv} reduces to the following procedure:
\begin{enumerate}
 \item Apply Algorithm~\ref{alg:rational_arnoldi} to compute orth. basis $U_{1,m}$ of $q_m(A_1)^{-1} \spK_m(A_1,\vb_1)$ and $G_{m} = U_{1,m}^* A_1 U_{1,m}$.
 \item Apply Algorithm~\ref{alg:rational_arnoldi}  to compute orth. basis $V_{2,m}$ of $\bar{q}_m(A_2)^{-*} \spK_m(A_2,\vc_2)$ and $H_{m} = V_{2,m}^* A_2^* V_{2,m}$.
 \item Solve Sylvester equation $G_m \tilde Z_m - \tilde Z_m H_m^* + U_{1,m}^* \vb_1 \vc_2^* V_{2,m} = 0$.
 \item Return approximate solution $Z_m = U_{1,m}\tilde Z_m V_{2,m}^*$
\end{enumerate}
This procedure turns out to be identical to existing rational Krylov subspace methods for Sylvester equations; see~\cite{Benner2009,Druskin2011} as well as~\cite{Simoncini2016} for additional references. In turn, the theory developed in this work can be used to bound the convergence of these methods via the best rational approximation of the sign function on $W(A_1)\cup W(-A_2)$. However, the bounds resulting from such an approach do not seem to offer advantages compared to existing bounds~\cite{Beckermann2011,Beckermann2019,Druskin2011} and we will therefore skip the details.

\section{Conclusions}\label{sec:conclusions}

The rational Krylov methods developed in this work constitute a fast way to approximate low-rank updates of the form $f(A+\vb\vc^*) - f(A)$, provided that shifted inverses with $A$ can be applied efficiently. Their computational cost is comparable to the application of existing rational Krylov methods for approximating $f(A)\vb$ and $f(A^*)\vc$.
This work has focussed on theoretical and algorithmic foundations. Future work will explore the application and the adaptation of our methods to specific problems in scientific computing and data science.

\begin{paragraph}{Acknowledgments} The authors gratefully acknowledge inspiring discussions with Stefano Massei, Vanni Noferini, and Ana \v{S}u\v{s}njara.

\end{paragraph}

\bibliographystyle{siam}
\bibliography{matrixfunctions}

\end{document}